\documentclass[oneside]{amsart}

\usepackage{amsthm}
\usepackage{amssymb}
\usepackage[usenames]{color} 
\usepackage{colortbl}
\usepackage[shortlabels]{enumitem}
\usepackage{caption}
\usepackage{subcaption}
\usepackage{graphicx}
\usepackage{tikz}
\usepackage{tikz-cd}

\usepackage[hidelinks]{hyperref}

\theoremstyle{definition}
 \newtheorem{theorem}{Theorem}
 \newtheorem{proposition}[equation]{Proposition}
 \newtheorem{definition}[equation]{Definition}
 \newtheorem{remark}[equation]{Remark}
 \newtheorem{lemma}[equation]{Lemma}
 \newtheorem{corollary}[equation]{Corollary}
 \newtheorem{example}[equation]{Example}
\newtheorem{question}[equation]{Question}

\newcommand\e{\varepsilon}

\newcommand{\Add}{\mathrm{Add}}

\newcommand\RR{\mathbb R}
\newcommand\A{\mathcal A}
\newcommand\ZZ{\mathbb Z}
\newcommand\QQ{\mathbb Q}
\newcommand\NN{\mathbb N}

\newcommand\B{\mathcal B}

\newcommand{\p}{{\bf p}}

\usepackage{amsaddr}
\title{Shrinking dynamics on multidimensional tropical series}

\author[N. Kalinin, M. Shkolnikov]{Nikita Kalinin, Mikhail Shkolnikov}

\begin{document}

\begin{abstract}
Let $\Omega\subset\RR^n$ be a compact convex domain.  An
$\Omega$-tropical series is a nonnegative, concave, integral-slope,
piecewise-affine function on $\Omega$ that vanishes on $\partial\Omega$.
For a finite set $P\subset\Omega^\circ$, we study the least such function
above prescribed initial data whose corner locus contains $P$.  It is
obtained by repeatedly applying one-point shrinking operators $G_p$.
We prove that every fair order of these operators stabilizes after finitely
many nontrivial steps.  We also describe an event-driven implementation that
records the lowest monomials at each point and updates only affected watcher
lists.  Finally, we show that, on every compact subset of $\Omega^\circ$, the
resulting dynamics can be approximated by a finite path whose intermediate
tropical hypersurfaces have only mild singularities on that compact set;
equivalently, the corresponding local cells of the dual regular subdivision
contain no lattice points other than their vertices.


\end{abstract}
\maketitle

\section{Introduction}

Fix a compact convex domain $\Omega\subset\RR^n$ and a finite set
$P\subset\Omega^\circ$.  The basic problem of this paper is to construct the
pointwise smallest nonnegative, concave, integral-slope piecewise-affine
function that vanishes on $\partial\Omega$ and is nonsmooth at every point of
$P$.  The corner locus of this function is a tropical analytic hypersurface
through the prescribed points.

The construction is dynamical.  If a current tropical series $f$ is smooth
at $p\in P$, exactly one affine monomial is minimal there.  The one-point
operator $G_p$ raises the coefficient of that monomial until it meets the
second-lowest monomial at $p$.  Repeating these local operations in a fair
order produces the least constrained series $G_Pf$.  This description is
elementary, but can be prohibitively slow with an the mean of its length being infinite even for two points \cite{shkolnikov2023relaxation}.

Our motivation comes from the Abelian sandpile model.  Tropical curves arise
in scaling limits and in the organization of recurrent patterns
\cite{firstsand,announce,us}.  The corresponding two-dimensional tropical
relaxation exhibits empirical power-law avalanche statistics
\cite{sandcomputation}.  Closely related coefficient deformations also occur
as ``breathing modes'' in \cite{vafa2012supersymmetric}.  The present paper
isolates the convex and tropical mechanism in arbitrary dimension. In addition, it serves as a preparation for stating and proving the generalization of another scaling limit for many-point perturbations of tropical sandpile with the macroscopic observable described by Monge-Amp\`ere equation, see \cite{KLSS26,LS26}.

The paper has three main contributions.  First, we establish the basic
multidimensional theory of tropical series on compact convex domains,
including the tropical distance function and the shrinking operators.
Second, we prove that every fair iteration of the operators $G_p$ reaches
$G_Pf$ after finitely many nontrivial updates.  Uniform convergence in the
plane was proved in \cite[Proposition~6.1]{us_series}, while finite
termination was asked in \cite[Question~7.2]{us_series}.  Third, we develop
an event-driven algorithm based on recorded minima, competitor caches, and
watcher lists.  In the generic two-monomial regime it replaces repeated
back-and-forth motion by a single certified block update.

We also study the geometry encountered along the dynamics.  A tropical
hypersurface is called \emph{mild} when each cell of its dual regular
subdivision contains no lattice points other than its vertices.  In dimension
two this allows smooth vertices and ordinary nodes, while excluding higher
lattice complexity.  We prove that the target dynamics can be approximated,
locally uniformly in $\Omega^\circ$, by a finite concatenation of partial
shrinking flows whose intermediate hypersurfaces are mild on any prescribed
compact set.  The use of compact subsets is essential: the approximating
functions live on inscribed rational-slope polytopes, so a global norm on two
different domains would not be well defined, and canonical monomials may
accumulate near the boundary.

Tropical series also occur as growth indicators in pluripotential theory
\cite{MR743626,kiselman2014questions,Abakumov2017}.  In one variable they are
related to ultradiscretization and tropical Nevanlinna theory
\cite{tohge2014order,MR2482129,MR2795727,korhonen2015tropical}, and to
weighted automata \cite{lahaye2015compositions,lombardy2006sequential}.
General references on tropical geometry include
\cite{BIMS,mikh2,Brug}.  Related planar wave-front dynamics are developed in
\cite{mikhalkin2023wave} in the convex domain case.

Section~\ref{sec_finite_algorithm} states the finite-stabilization theorem
and presents the accelerated computation early.  Sections~\ref{sec_series}
and~\ref{sec_omegaseries} develop tropical and $\Omega$-tropical series.
Sections~\ref{sec_distance}--\ref{sec_dynamic} treat tropical distance,
shrinking operators, and their dynamics.  Mild singularities are discussed
in Section~\ref{sec_mild}, polyhedral exhaustion in
Section~\ref{sec_exhausting}, and the approximation theorem in
Section~\ref{sec_main}.

It seems to us that most, if not all, of the results of the present paper extend to the case of tropical series which are \emph{not} required to vanish at the boundary. The potential necessity to do so is to have a solver for Monge-Amp\`ere equation with boundary conditions different from that of Dirichlet. In fact, our principal motivation for the fast relaxation algorithm is a preparation for implementing that in hope it would give a competitive alternative to the existing solvers. On the other hand, for practical applications, it is desirable to extend the class of PDEs we are currently having access to. We believe that a recent work \cite{Varypaev25} provides one conceptual hint for realization such an extension. Another remark we would like to make is that working with mere tropical power series, i.e. infinite-expression analogues of tropical polynomials, restricts the potential scaling limits of such to concave functions -- to overcome this, one could work with a broader class of functions, for instance tropical rational functions, or their infinite-expression analogues presented by differences (i.e. tropical fractions) of two tropical power series.

\textbf{Acknowledgments.} We thank Andrea Sportiello for sharing his insights
on perturbative regimes of the Abelian sandpile model, which initiated our
work on sandpiles first in dimension two and now in higher dimensions. We express our immense gratitude to all the people who support and empathize with our work, from our teachers, whose opinions and views still shape the flow of our thought, to the youngest students, whose enthusiasm and diligence keep us going.

\section{Finite stabilization and accelerated computation}
\label{sec_finite_algorithm}

We state the finite-stabilization result near the beginning because it is both
a structural property of the shrinking dynamics and the basis of the
algorithm below.  The definitions and elementary properties used here are
developed in Sections~\ref{sec_series}-\ref{sec_gp}.

Let $\Omega\subset\RR^n$ be a compact convex set with nonempty interior, let
$P\subset\Omega^\circ$ be finite, and let $f$ be an $\Omega$-tropical
series.  We write every series in its full canonical form
\begin{equation}\label{eq_finite_canonical_form}
  h(z)=\inf_{q\in\ZZ^n}\bigl(q\cdot z+a_q(h)\bigr),
  \qquad
  a_q(h)=\sup_{z\in\Omega}\bigl(h(z)-q\cdot z\bigr).
\end{equation}
Recall that $G_Pf$ is the pointwise minimal $\Omega$-tropical series above
$f$ whose corner locus contains $P$.  If $P=\{p\}$, we write $G_p$.

Choose a sequence $p_0,p_1,\ldots\in P$ in which every point of $P$ occurs
infinitely often, and set
\begin{equation}\label{eq_finite_iteration}
  f_0=f,
  \qquad
  f_{m+1}=G_{p_m}f_m.
\end{equation}
We call such an order \emph{fair}.  Lemma~\ref{lem_singlegp} says that a
nontrivial $G_p$ raises the coefficient of the unique minimizing monomial at
$p$ until it ties the second-lowest monomial.  Proposition~\ref{prop_gpsconvergence}
says that the sequence~\eqref{eq_finite_iteration} converges uniformly to
$G_Pf$.

\subsection{Canonical coefficients near the limit}

\begin{lemma}[Uniform Lipschitz control of canonical coefficients]
\label{lem_finite_coefficient_lipschitz}
For any two $\Omega$-tropical series $h,k$ and every $q\in\ZZ^n$,
\[
  |a_q(h)-a_q(k)|\leq \lVert h-k\rVert_\infty.
\]
Consequently, the estimate is uniform over all $q\in\ZZ^n$.
\end{lemma}

\begin{proof}
If $\lVert h-k\rVert_\infty\leq\varepsilon$, then
\[
  k(z)-q\cdot z-\varepsilon
  \leq h(z)-q\cdot z
  \leq k(z)-q\cdot z+\varepsilon.
\]
Taking the supremum over $z\in\Omega$ and using
\eqref{eq_finite_canonical_form} proves the assertion.
\end{proof}

Fix $F=G_Pf$.  Monotonicity of the shrinking operators and $G_pF=F$ for
every $p\in P$ give
\[
  f_m\leq F\qquad(m\geq0).
\]
Define the canonical-coefficient deficits
\begin{equation}\label{eq_finite_deficit}
  d_q^{(m)}=a_q(F)-a_q(f_m),
  \qquad q\in\ZZ^n.
\end{equation}
By Lemma~\ref{lem_finite_coefficient_lipschitz} and uniform convergence,
\begin{equation}\label{eq_finite_deficit_bound}
  0\leq d_q^{(m)}
  \leq \lVert F-f_m\rVert_\infty
  \longrightarrow0,
\end{equation}
uniformly in $q$.

For $p\in P$, let
\begin{equation}\label{eq_finite_active_set}
  \mathcal A_p
  =\bigl\{q\in\ZZ^n:
       a_q(F)+q\cdot p=F(p)\bigr\}
\end{equation}
be the set of canonical slopes active at $p$ in the limiting series.  Since
$p\in C(F)$, the set $\mathcal A_p$ contains at least two elements.

\begin{lemma}[Positive terminal gap]
\label{lem_finite_terminal_gap}
For every $p\in P$, the set $\mathcal A_p$ is finite and there exists
$\gamma_p>0$ such that
\begin{equation}\label{eq_finite_gap}
  a_q(F)+q\cdot p\geq F(p)+\gamma_p
  \qquad\text{for all }q\notin\mathcal A_p.
\end{equation}
\end{lemma}

\begin{proof}
Choose $r>0$ such that the closed ball $\overline B(p,r)$ lies in
$\Omega^\circ$.  Every canonical affine term
\[
  L_q(z)=a_q(F)+q\cdot z
\]
majorizes $F$ and hence is nonnegative on $\Omega$.  For $q\neq0$, evaluate
$L_q$ at $p-rq/|q|$ to obtain
\[
  0\leq L_q\left(p-r\frac{q}{|q|}\right)
  =L_q(p)-r|q|.
\]
Thus $L_q(p)\geq r|q|$.  It follows that only finitely many
$q\in\ZZ^n$ satisfy $L_q(p)\leq F(p)+1$.

In particular, $\mathcal A_p$ is finite.  Among the finitely many slopes
satisfying $L_q(p)\leq F(p)+1$ but not belonging to $\mathcal A_p$, every
difference $L_q(p)-F(p)$ is strictly positive.  Taking the minimum of these
differences and $1$ gives the required $\gamma_p$.
\end{proof}

\subsection{Finite stabilization}

\begin{lemma}[Deficit-copying regime]
\label{lem_finite_deficit_copying}
There exists $m_0<\infty$ such that, for every $m\geq m_0$ and every
$p\in P$, the following holds.  If $G_p f_m\neq f_m$, then there are
$q,r\in\mathcal A_p$ such that
\begin{align*}
  d_q^{(m)}&>d_r^{(m)},\\
  d_q(G_pf_m)&=d_r^{(m)},\\
  d_s(G_pf_m)&=d_s^{(m)}\qquad(s\neq q),
\end{align*}
where $d_s(G_pf_m)=a_s(F)-a_s(G_pf_m)$.  Thus a nontrivial shrinking
operator replaces one deficit by a strictly smaller deficit already present
at another terminally active slope.
\end{lemma}

\begin{proof}
Let $\gamma=\min_{p\in P}\gamma_p>0$.  By
\eqref{eq_finite_deficit_bound}, choose $m_0$ such that
\[
  d_q^{(m)}<\frac{\gamma}{3}
  \qquad\text{for every }q\in\ZZ^n\text{ and }m\geq m_0.
\]
Fix $p\in P$ and $m\geq m_0$.  For $q\in\mathcal A_p$, the value at $p$
of the corresponding monomial of $f_m$ is
\begin{equation}\label{eq_finite_active_current_value}
  a_q(f_m)+q\cdot p=F(p)-d_q^{(m)}\leq F(p).
\end{equation}
For $q\notin\mathcal A_p$, Lemma~\ref{lem_finite_terminal_gap} gives
\begin{equation}\label{eq_finite_inactive_current_value}
  a_q(f_m)+q\cdot p
  \geq F(p)+\gamma_p-d_q^{(m)}
  >F(p).
\end{equation}
Because $|\mathcal A_p|\geq2$, both the lowest and the second-lowest
monomials of $f_m$ at $p$ therefore have slopes in $\mathcal A_p$.

Suppose that $G_pf_m\neq f_m$.  Then $f_m$ has a unique minimizing slope
$q\in\mathcal A_p$ at $p$.  By
\eqref{eq_finite_active_current_value}, $d_q^{(m)}$ is the unique largest
deficit in $\mathcal A_p$.  Choose $r\in\mathcal A_p\setminus\{q\}$ with
the second-largest deficit.  The shrinking operator raises the
$q$-coefficient until the $q$- and $r$-monomials tie, so
\[
  a_q(G_pf_m)+q\cdot p
  =a_r(f_m)+r\cdot p
  =F(p)-d_r^{(m)}.
\]
On the other hand, $a_q(F)+q\cdot p=F(p)$, and hence
\[
  a_q(F)-a_q(G_pf_m)=d_r^{(m)}.
\]

It remains to check that these are canonical coefficients, rather than merely
coefficients in a noncanonical presentation.  For $s\neq q$, we have
$G_pf_m\geq f_m$, so $a_s(G_pf_m)\geq a_s(f_m)$.  The unchanged
$s$-monomial remains in the displayed minimum defining $G_pf_m$, giving the
reverse inequality $a_s(G_pf_m)\leq a_s(f_m)$.  Thus all such canonical
coefficients are unchanged.  The raised $q$-monomial attains the new minimum
at $p$, so its new coefficient is canonical as well.  Finally, uniqueness of
the original minimum implies $d_q^{(m)}>d_r^{(m)}$.
\end{proof}

\begin{theorem}[Finite stabilization]
\label{thm_finite_stabilization}
Let $P\subset\Omega^\circ$ be finite, and let $(p_m)_{m\geq0}$ be a fair
sequence in $P$.  Then the iteration~\eqref{eq_finite_iteration} stabilizes
after finitely many applications: there exists $M<\infty$ such that
\[
  f_m=G_Pf\qquad\text{for all }m\geq M.
\]
In particular, this holds for $f=0_\Omega$.
\end{theorem}

\begin{proof}
Let
\[
  \mathcal A=\bigcup_{p\in P}\mathcal A_p.
\]
This is a finite set by Lemma~\ref{lem_finite_terminal_gap}.  At time $m_0$
from Lemma~\ref{lem_finite_deficit_copying}, form the finite set of deficit
values
\[
  D=\{d_q^{(m_0)}:q\in\mathcal A\}
   =\{\lambda_0<\lambda_1<\cdots<\lambda_s\}.
\]
Lemma~\ref{lem_finite_deficit_copying} shows inductively that after time
$m_0$ every deficit indexed by $\mathcal A$ remains in $D$.  Moreover, every
nontrivial shrinking operator changes a single deficit from some $\lambda_i$
to a strictly smaller $\lambda_j$.

Define the nonnegative integer potential
\[
  \Phi_m
  =\sum_{q\in\mathcal A}\operatorname{rank}_D(d_q^{(m)}),
  \qquad
  \operatorname{rank}_D(\lambda_i)=i.
\]
Every nontrivial shrinking operator after $m_0$ strictly decreases $\Phi_m$.
Consequently there can be at most
$\Phi_{m_0}\leq|\mathcal A|s$ further nontrivial steps.  The sequence
$(f_m)$ therefore has a constant tail; denote its value by $\overline F$.

Fairness implies that $\overline F$ is nonsmooth at every point of $P$:
otherwise a later occurrence of a smooth point $p$ would make
$G_p\overline F\neq\overline F$.  Since $\overline F\geq f$, it follows
that $\overline F\in V(\Omega,P,f)$, and minimality gives
\[
  G_Pf\leq\overline F.
\]
Every iterate is bounded above by $G_Pf$, so the reverse inequality also
holds.  Hence $\overline F=G_Pf$, completing the proof.
\end{proof}

\begin{remark}
The proof is qualitative: the entry time $m_0$ is obtained from uniform
convergence and need not admit a useful configuration-independent bound.
Thus exact dynamics always terminates, while floating-point implementations
may display very long, apparently Zeno-like transients.
\end{remark}

\begin{remark}
Nothing in the deficit argument is specific to dimension two.  It applies in
every finite dimension for which the canonical coefficient formula, the
one-coefficient description of $G_p$, and uniform convergence to $G_Pf$
hold.  For an unbounded domain with only locally uniform convergence,
additional control of the full family of canonical coefficients would be
required.
\end{remark}

\subsection{Algorithmic realization}

We now spell out the algorithmic content of the preceding discussion.  The
goal is to construct $G_Pf$, or a close numerical approximation to it, by
modifying the canonical coefficients of a tropical series.  Write the current
series as
\[
  h(z)=\min_{q\in\Lambda}\bigl(a_q+q\cdot z\bigr),
\]
where $\Lambda\subset\ZZ^n$ is the finite set of slopes currently retained by
the computation.  In an exact implementation one enlarges $\Lambda$ whenever
a new slope can become relevant; in a truncated implementation $\Lambda$ is a
large finite box of slopes.

\subsubsection*{The naive algorithm}

The direct implementation of the one-point dynamics is the following.

\begin{enumerate}
  \item Start with $h=f$ and a finite slope set $\Lambda$.

  \item Sweep through the points $p\in P$ in some fair order.

  \item At the chosen point $p$, compute all values
  \[
    V_q(p)=a_q+q\cdot p,
    \qquad q\in\Lambda,
  \]
  and find the smallest and second-smallest values.

  \item If the smallest value is attained at least twice, then
  $p\in C(h)$ and nothing is done.

  \item If the minimum is unique, say it is attained at $q$, and the
  second-lowest monomial is $r$, replace
  \[
    a_q\quad\text{by}\quad a_q+\bigl(V_r(p)-V_q(p)\bigr).
  \]
  This is exactly the one-point operator $G_p$: it raises the unique
  minimizing monomial until it ties the second-lowest monomial at $p$.

  \item After the update, recompute the relevant minima at all points of $P$
  and continue until every point lies in the corner locus.
\end{enumerate}

This algorithm is conceptually simple but computationally wasteful.  A single
coefficient update can only affect points at which that coefficient is among
the currently relevant candidates, yet the naive procedure scans all points
and all monomials after every one-point update.

\subsubsection*{Recorded minima and watchers}

The faster version keeps local memory at each dropped point.  For every
$p\in P$ record
\[
  M(p)=\{q\in\Lambda:
       V_q(p)=\min_{r\in\Lambda}V_r(p)\},
\]
the set of currently minimal monomials.  Also retain a short ordered cache
$N(p)$ of the next-lowest monomials outside $M(p)$.  Generically it is enough
to store the two lowest monomials and the next competitor; the cache is
extended lazily if all its entries later enter the same moving block.

For each monomial $q$ keep a watcher list
\[
  W_q=\{p\in P:q\in M(p)\cup N(p)\}.
\]
Thus a point remembers its minimal monomials and its nearest competitors,
while every recorded monomial remembers the points at which it is recorded.
If a coefficient $a_q$ changes, only points in $W_q$ need to have their local
data reconsidered.  Watcher lists may be append-only: when an old entry is
encountered, it is ignored if $q$ is no longer recorded at that point.

The competitor cache gives the first obstruction to moving a whole block of
minimal monomials together.  If all monomials in $M(p)$ are raised by the
same amount, then $p$ remains on the corner locus until their common value
reaches the first monomial outside the moving block.

\subsubsection*{The avalanche update}

Suppose a point $p_0$ is not on the corner locus.  Then
$M(p_0)=\{q_0\}$ is a singleton.  The ordinary operator $G_{p_0}$ would
raise $a_{q_0}$ until $q_0$ meets the second-lowest monomial.  Raising this
coefficient may, however, destroy ties at other points that use $q_0$ as one
of their minimal monomials.  The avalanche algorithm first finds all
coefficients that are forced to move together.  We first describe the generic
case in which every affected point already on the corner locus has exactly
two minimal monomials.  If a point with three or more minima is encountered,
one either solves the corresponding finite event linear program or falls
back to the exact one-point update; blindly moving every monomial in a
higher-order tie can overshoot the least solution.

Start with $S=\{q_0\}$ and inspect the watcher lists of the monomials in
$S$.  Whenever a point $p$ satisfies
\[
  M(p)\cap S\neq\varnothing
  \qquad\text{and}\qquad
  M(p)\not\subset S,
\]
add all of $M(p)$ to $S$ and inspect the watcher lists of the newly added
monomials.  In other words, if one monomial in a currently tied minimum is
forced upward, all monomials in that tied minimum must move with it; otherwise
$p$ would immediately leave the corner locus.  This closure process stops
because only finitely many monomials and points are recorded.

Let
\[
  R(S)=\{p\in P:M(p)\cap S\neq\varnothing\}.
\]
After closure, $M(p)\subset S$ for every $p\in R(S)$.  For such a point,
let
\[
  n_S(p)\in\operatorname*{arg\,min}_{q\in\Lambda\setminus S}V_q(p)
\]
be its first competitor outside the moving block.  Usually $n_S(p)$ is
already in the short cache $N(p)$.  If the cached competitors all belong to
$S$, scan further candidates at this point and extend the cache.  This detail
is essential when two nearby points make an edge move back and forth: the two
edge monomials belong to $S$, and the event is governed by the next monomial
outside that pair.

The largest common increment that preserves all existing ties in the affected
component is
\[
  \delta_*=\min_{p\in R(S)}
  \left(
    V_{n_S(p)}(p)-\min_{q\in M(p)}V_q(p)
  \right).
\]
Perform the simultaneous update
\[
  a_q\leftarrow a_q+\delta_*,
  \qquad q\in S.
\]
No point already on the corner locus is lost: its minimal monomials are either
all outside $S$ or all inside $S$.  At least one point of $R(S)$ reaches a new
event, where the first monomial outside $S$ joins its minimum.
Figure~\ref{fig_avalanche_three_steps} illustrates three successive updates in
$\RR^2$, including the release of points at higher-order ties.

\clearpage
\begin{figure}[p]
\centering
\captionsetup[subfigure]{justification=centering}
\begin{subfigure}[t]{0.32\linewidth}
  \centering
  \resizebox{\linewidth}{!}{
\begin{tikzpicture}[x=1cm,y=1cm,line cap=round,line join=round]
\definecolor{avred}{RGB}{226,78,72}
\definecolor{avblue}{RGB}{22,119,184}
\definecolor{avgreen}{RGB}{73,133,107}
\definecolor{avgray}{RGB}{160,166,173}
\path[use as bounding box] (-3,-2.8) rectangle (3,2.6);
\begin{scope}
\clip (-3,-2.8) rectangle (3,2.6);
\fill[avred,opacity=.16] (0.0000000,-1.6666667) -- (0.0000000,0.0000000) -- (-0.2000000,0.2000000) -- (-1.6000000,0.2000000) -- (-2.5000000,-2.5000000) -- cycle;
\fill[avred,opacity=.16] (2.5000000,-2.5000000) -- (1.6666667,0.0000000) -- (0.0000000,0.0000000) -- (0.0000000,-1.6666667) -- cycle;
\fill[avred,opacity=.16] (-0.2000000,1.2333333) -- (-0.2000000,0.2000000) -- (0.0000000,0.0000000) -- (1.6666667,0.0000000) -- (2.3625000,2.0875000) -- cycle;
\fill[avred,opacity=.16] (-2.1625000,1.8875000) -- (-1.6000000,0.2000000) -- (-0.2000000,0.2000000) -- (-0.2000000,1.2333333) -- cycle;
\draw[avgray,dash pattern=on 2.4pt off 2.4pt,line width=.7pt] (0.0000000,-1.6666667) -- (0.0000000,0.0000000);
\draw[avgray,dash pattern=on 2.4pt off 2.4pt,line width=.7pt] (0.0000000,0.0000000) -- (-0.2000000,0.2000000);
\draw[avgray,dash pattern=on 2.4pt off 2.4pt,line width=.7pt] (-0.2000000,0.2000000) -- (-1.6000000,0.2000000);
\draw[avgray,dash pattern=on 2.4pt off 2.4pt,line width=.7pt] (-1.6000000,0.2000000) -- (-2.5000000,-2.5000000);
\draw[avgray,dash pattern=on 2.4pt off 2.4pt,line width=.7pt] (0.0000000,-1.6666667) -- (-2.5000000,-2.5000000);
\draw[avgray,dash pattern=on 2.4pt off 2.4pt,line width=.7pt] (1.6666667,0.0000000) -- (0.0000000,0.0000000);
\draw[avgray,dash pattern=on 2.4pt off 2.4pt,line width=.7pt] (2.5000000,-2.5000000) -- (1.6666667,0.0000000);
\draw[avgray,dash pattern=on 2.4pt off 2.4pt,line width=.7pt] (2.5000000,-2.5000000) -- (0.0000000,-1.6666667);
\draw[avgray,dash pattern=on 2.4pt off 2.4pt,line width=.7pt] (-0.2000000,1.2333333) -- (-0.2000000,0.2000000);
\draw[avgray,dash pattern=on 2.4pt off 2.4pt,line width=.7pt] (1.6666667,0.0000000) -- (2.3625000,2.0875000);
\draw[avgray,dash pattern=on 2.4pt off 2.4pt,line width=.7pt] (-0.2000000,1.2333333) -- (2.3625000,2.0875000);
\draw[avgray,dash pattern=on 2.4pt off 2.4pt,line width=.7pt] (-2.1625000,1.8875000) -- (-1.6000000,0.2000000);
\draw[avgray,dash pattern=on 2.4pt off 2.4pt,line width=.7pt] (-2.1625000,1.8875000) -- (-0.2000000,1.2333333);
\draw[avgray,dash pattern=on 2.4pt off 2.4pt,line width=.7pt] (-2.8000000,-2.8000000) -- (-2.5000000,-2.5000000);
\draw[avgray,dash pattern=on 2.4pt off 2.4pt,line width=.7pt] (-2.1625000,1.8875000) -- (-2.8750000,2.6000000);
\draw[avgray,dash pattern=on 2.4pt off 2.4pt,line width=.7pt] (2.8000000,-2.8000000) -- (2.5000000,-2.5000000);
\draw[avgray,dash pattern=on 2.4pt off 2.4pt,line width=.7pt] (2.8750000,2.6000000) -- (2.3625000,2.0875000);
\draw[black!80,line width=.85pt] (0.0000000,-1.3333333) -- (0.0000000,0.0000000);
\draw[black!80,line width=.85pt] (0.0000000,0.0000000) -- (-0.2000000,0.2000000);
\draw[black!80,line width=.85pt] (-0.2000000,0.2000000) -- (-1.2666667,0.2000000);
\draw[avblue,line width=1.15pt] (-1.2666667,0.2000000) -- (-2.0000000,-2.0000000);
\draw[avblue,line width=1.15pt] (0.0000000,-1.3333333) -- (-2.0000000,-2.0000000);
\draw[black!80,line width=.85pt] (1.3333333,0.0000000) -- (0.0000000,0.0000000);
\draw[avblue,line width=1.15pt] (2.0000000,-2.0000000) -- (1.3333333,0.0000000);
\draw[avblue,line width=1.15pt] (2.0000000,-2.0000000) -- (0.0000000,-1.3333333);
\draw[black!80,line width=.85pt] (-0.2000000,0.9000000) -- (-0.2000000,0.2000000);
\draw[avblue,line width=1.15pt] (1.3333333,0.0000000) -- (1.8625000,1.5875000);
\draw[avblue,line width=1.15pt] (-0.2000000,0.9000000) -- (1.8625000,1.5875000);
\draw[avblue,line width=1.15pt] (-1.6625000,1.3875000) -- (-1.2666667,0.2000000);
\draw[avblue,line width=1.15pt] (-1.6625000,1.3875000) -- (-0.2000000,0.9000000);
\draw[black!80,line width=.85pt] (-2.8000000,-2.8000000) -- (-2.0000000,-2.0000000);
\draw[black!80,line width=.85pt] (-1.6625000,1.3875000) -- (-2.8750000,2.6000000);
\draw[black!80,line width=.85pt] (2.8000000,-2.8000000) -- (2.0000000,-2.0000000);
\draw[black!80,line width=.85pt] (2.8750000,2.6000000) -- (1.8625000,1.5875000);
\end{scope}
\path (-1.1000000,-1.2000000) node[text=avred!70!black,font=\fontsize{10}{11}\selectfont,inner sep=0pt] {$F_0$};
\path (1.0500000,-1.1000000) node[text=avred!70!black,font=\fontsize{10}{11}\selectfont,inner sep=0pt] {$F_1$};
\path (0.9000000,0.6800000) node[text=avred!70!black,font=\fontsize{10}{11}\selectfont,inner sep=0pt] {$F_2$};
\path (-0.9000000,0.7000000) node[text=avred!70!black,font=\fontsize{10}{11}\selectfont,inner sep=0pt] {$F_3$};
\filldraw[fill=black,draw=white,line width=.3pt] (0.0000000,-0.7000000) circle[radius=1.15pt];
\path (0.1800000,-0.9600000) node[fill=white,fill opacity=.85,text opacity=1,text=black,font=\fontsize{12}{13}\selectfont,inner sep=.3pt] {$p_1$};
\filldraw[fill=black,draw=white,line width=.3pt] (0.8000000,0.0000000) circle[radius=1.15pt];
\path (1.1000000,-0.1400000) node[fill=white,fill opacity=.85,text opacity=1,text=black,font=\fontsize{12}{13}\selectfont,inner sep=.3pt] {$p_2$};
\filldraw[fill=white,draw=avblue,line width=.8pt] (-0.2000000,0.9000000) circle[radius=2.3pt];
\filldraw[fill=avblue,draw=white,line width=.3pt] (-0.2000000,0.9000000) circle[radius=1.15pt];
\path (-0.4900000,1.1600000) node[fill=white,fill opacity=.85,text opacity=1,text=avblue,font=\fontsize{12}{13}\selectfont,inner sep=.3pt] {$p_3$};
\filldraw[fill=avred!80!black,draw=white,line width=.3pt] (-0.8000000,-0.6500000) circle[radius=1.15pt];
\path (-1.0900000,-0.9400000) node[fill=white,fill opacity=.85,text opacity=1,text=avred!80!black,font=\fontsize{12}{13}\selectfont,inner sep=.3pt] {$p_4$};
\end{tikzpicture}}
  \caption{$p_3$ releases;\\ $\Delta_1=1$.}
\end{subfigure}\hfill
\begin{subfigure}[t]{0.32\linewidth}
  \centering
  \resizebox{\linewidth}{!}{
\begin{tikzpicture}[x=1cm,y=1cm,line cap=round,line join=round]
\definecolor{avred}{RGB}{226,78,72}
\definecolor{avblue}{RGB}{22,119,184}
\definecolor{avgreen}{RGB}{73,133,107}
\definecolor{avgray}{RGB}{160,166,173}
\path[use as bounding box] (-3,-2.8) rectangle (3,2.6);
\begin{scope}
\clip (-3,-2.8) rectangle (3,2.6);
\fill[avred,opacity=.16] (0.0000000,-1.3333333) -- (0.0000000,0.0000000) -- (-0.2000000,0.2000000) -- (-1.2666667,0.2000000) -- (-2.0000000,-2.0000000) -- cycle;
\fill[avred,opacity=.16] (2.0000000,-2.0000000) -- (1.3333333,0.0000000) -- (0.0000000,0.0000000) -- (0.0000000,-1.3333333) -- cycle;
\fill[avred,opacity=.16] (-0.2000000,0.9000000) -- (-0.2000000,0.2000000) -- (0.0000000,0.0000000) -- (1.3333333,0.0000000) -- (1.8625000,1.5875000) -- cycle;
\draw[avgray,dash pattern=on 2.4pt off 2.4pt,line width=.7pt] (0.0000000,-1.3333333) -- (0.0000000,0.0000000);
\draw[avgray,dash pattern=on 2.4pt off 2.4pt,line width=.7pt] (0.0000000,0.0000000) -- (-0.2000000,0.2000000);
\draw[avgray,dash pattern=on 2.4pt off 2.4pt,line width=.7pt] (-0.2000000,0.2000000) -- (-1.2666667,0.2000000);
\draw[avgray,dash pattern=on 2.4pt off 2.4pt,line width=.7pt] (-1.2666667,0.2000000) -- (-2.0000000,-2.0000000);
\draw[avgray,dash pattern=on 2.4pt off 2.4pt,line width=.7pt] (0.0000000,-1.3333333) -- (-2.0000000,-2.0000000);
\draw[avgray,dash pattern=on 2.4pt off 2.4pt,line width=.7pt] (1.3333333,0.0000000) -- (0.0000000,0.0000000);
\draw[avgray,dash pattern=on 2.4pt off 2.4pt,line width=.7pt] (2.0000000,-2.0000000) -- (1.3333333,0.0000000);
\draw[avgray,dash pattern=on 2.4pt off 2.4pt,line width=.7pt] (2.0000000,-2.0000000) -- (0.0000000,-1.3333333);
\draw[avgray,dash pattern=on 2.4pt off 2.4pt,line width=.7pt] (-0.2000000,0.9000000) -- (-0.2000000,0.2000000);
\draw[avgray,dash pattern=on 2.4pt off 2.4pt,line width=.7pt] (1.3333333,0.0000000) -- (1.8625000,1.5875000);
\draw[avgray,dash pattern=on 2.4pt off 2.4pt,line width=.7pt] (-0.2000000,0.9000000) -- (1.8625000,1.5875000);
\draw[avgray,dash pattern=on 2.4pt off 2.4pt,line width=.7pt] (-1.6625000,1.3875000) -- (-1.2666667,0.2000000);
\draw[avgray,dash pattern=on 2.4pt off 2.4pt,line width=.7pt] (-1.6625000,1.3875000) -- (-0.2000000,0.9000000);
\draw[avgray,dash pattern=on 2.4pt off 2.4pt,line width=.7pt] (-2.8000000,-2.8000000) -- (-2.0000000,-2.0000000);
\draw[avgray,dash pattern=on 2.4pt off 2.4pt,line width=.7pt] (-1.6625000,1.3875000) -- (-2.8750000,2.6000000);
\draw[avgray,dash pattern=on 2.4pt off 2.4pt,line width=.7pt] (2.8000000,-2.8000000) -- (2.0000000,-2.0000000);
\draw[avgray,dash pattern=on 2.4pt off 2.4pt,line width=.7pt] (2.8750000,2.6000000) -- (1.8625000,1.5875000);
\draw[black!80,line width=.85pt] (0.0000000,-0.8000000) -- (0.0000000,-0.6000000);
\draw[avblue,line width=1.15pt] (0.0000000,-0.6000000) -- (-1.0000000,-0.6000000);
\draw[avblue,line width=1.15pt] (-1.0000000,-0.6000000) -- (-1.2000000,-1.2000000);
\draw[avblue,line width=1.15pt] (0.0000000,-0.8000000) -- (-1.2000000,-1.2000000);
\draw[black!80,line width=.85pt] (0.8000000,0.0000000) -- (0.6000000,0.0000000);
\draw[avblue,line width=1.15pt] (0.6000000,0.0000000) -- (0.0000000,-0.6000000);
\draw[avblue,line width=1.15pt] (1.2000000,-1.2000000) -- (0.8000000,0.0000000);
\draw[avblue,line width=1.15pt] (1.2000000,-1.2000000) -- (0.0000000,-0.8000000);
\draw[avblue,line width=1.15pt] (0.6000000,0.6333333) -- (0.6000000,0.0000000);
\draw[avblue,line width=1.15pt] (0.8000000,0.0000000) -- (1.0625000,0.7875000);
\draw[avblue,line width=1.15pt] (0.6000000,0.6333333) -- (1.0625000,0.7875000);
\draw[black!80,line width=.85pt] (-1.6625000,1.3875000) -- (-1.0000000,-0.6000000);
\draw[black!80,line width=.85pt] (-1.6625000,1.3875000) -- (0.6000000,0.6333333);
\draw[black!80,line width=.85pt] (-2.8000000,-2.8000000) -- (-1.2000000,-1.2000000);
\draw[black!80,line width=.85pt] (-1.6625000,1.3875000) -- (-2.8750000,2.6000000);
\draw[black!80,line width=.85pt] (2.8000000,-2.8000000) -- (1.2000000,-1.2000000);
\draw[black!80,line width=.85pt] (2.8750000,2.6000000) -- (1.0625000,0.7875000);
\end{scope}
\path (-0.4000000,-0.8000000) node[text=avred!70!black,font=\fontsize{10}{11}\selectfont,inner sep=0pt] {$F_0$};
\path (0.6200000,-0.6200000) node[text=avred!70!black,font=\fontsize{10}{11}\selectfont,inner sep=0pt] {$F_1$};
\path (0.8000000,0.3900000) node[text=avred!70!black,font=\fontsize{10}{11}\selectfont,inner sep=0pt] {$F_2$};
\path (-0.7500000,0.2800000) node[text=black!50,font=\fontsize{10}{11}\selectfont,inner sep=0pt] {$F_3$};
\filldraw[fill=black,draw=white,line width=.3pt] (0.0000000,-0.7000000) circle[radius=1.15pt];
\path (0.1800000,-0.9600000) node[fill=white,fill opacity=.85,text opacity=1,text=black,font=\fontsize{12}{13}\selectfont,inner sep=.3pt] {$p_1$};
\filldraw[fill=white,draw=avblue,line width=.8pt] (0.8000000,0.0000000) circle[radius=2.3pt];
\filldraw[fill=avblue,draw=white,line width=.3pt] (0.8000000,0.0000000) circle[radius=1.15pt];
\path (1.1000000,-0.1400000) node[fill=white,fill opacity=.85,text opacity=1,text=avblue,font=\fontsize{12}{13}\selectfont,inner sep=.3pt] {$p_2$};
\filldraw[fill=avgreen,draw=white,line width=.3pt] (-0.2000000,0.9000000) circle[radius=1.15pt];
\path (-0.4900000,1.1600000) node[fill=white,fill opacity=.85,text opacity=1,text=avgreen,font=\fontsize{12}{13}\selectfont,inner sep=.3pt] {$p_3$};
\filldraw[fill=avred!80!black,draw=white,line width=.3pt] (-0.8000000,-0.6500000) circle[radius=1.15pt];
\path (-1.0900000,-0.9400000) node[fill=white,fill opacity=.85,text opacity=1,text=avred!80!black,font=\fontsize{12}{13}\selectfont,inner sep=.3pt] {$p_4$};
\end{tikzpicture}}
  \caption{$p_2$ releases;\\ $\Delta_2=1.6$.}
\end{subfigure}\hfill
\begin{subfigure}[t]{0.32\linewidth}
  \centering
  \resizebox{\linewidth}{!}{
\begin{tikzpicture}[x=1cm,y=1cm,line cap=round,line join=round]
\definecolor{avred}{RGB}{226,78,72}
\definecolor{avblue}{RGB}{22,119,184}
\definecolor{avgreen}{RGB}{73,133,107}
\definecolor{avgray}{RGB}{160,166,173}
\path[use as bounding box] (-3,-2.8) rectangle (3,2.6);
\begin{scope}
\clip (-3,-2.8) rectangle (3,2.6);
\fill[avred,opacity=.16] (0.0000000,-0.8000000) -- (0.0000000,-0.6000000) -- (-1.0000000,-0.6000000) -- (-1.2000000,-1.2000000) -- cycle;
\fill[avred,opacity=.16] (1.2000000,-1.2000000) -- (0.8000000,0.0000000) -- (0.6000000,0.0000000) -- (0.0000000,-0.6000000) -- (0.0000000,-0.8000000) -- cycle;
\draw[avgray,dash pattern=on 2.4pt off 2.4pt,line width=.7pt] (0.0000000,-0.8000000) -- (0.0000000,-0.6000000);
\draw[avgray,dash pattern=on 2.4pt off 2.4pt,line width=.7pt] (0.0000000,-0.6000000) -- (-1.0000000,-0.6000000);
\draw[avgray,dash pattern=on 2.4pt off 2.4pt,line width=.7pt] (-1.0000000,-0.6000000) -- (-1.2000000,-1.2000000);
\draw[avgray,dash pattern=on 2.4pt off 2.4pt,line width=.7pt] (0.0000000,-0.8000000) -- (-1.2000000,-1.2000000);
\draw[avgray,dash pattern=on 2.4pt off 2.4pt,line width=.7pt] (0.8000000,0.0000000) -- (0.6000000,0.0000000);
\draw[avgray,dash pattern=on 2.4pt off 2.4pt,line width=.7pt] (0.6000000,0.0000000) -- (0.0000000,-0.6000000);
\draw[avgray,dash pattern=on 2.4pt off 2.4pt,line width=.7pt] (1.2000000,-1.2000000) -- (0.8000000,0.0000000);
\draw[avgray,dash pattern=on 2.4pt off 2.4pt,line width=.7pt] (1.2000000,-1.2000000) -- (0.0000000,-0.8000000);
\draw[avgray,dash pattern=on 2.4pt off 2.4pt,line width=.7pt] (0.6000000,0.6333333) -- (0.6000000,0.0000000);
\draw[avgray,dash pattern=on 2.4pt off 2.4pt,line width=.7pt] (0.8000000,0.0000000) -- (1.0625000,0.7875000);
\draw[avgray,dash pattern=on 2.4pt off 2.4pt,line width=.7pt] (0.6000000,0.6333333) -- (1.0625000,0.7875000);
\draw[avgray,dash pattern=on 2.4pt off 2.4pt,line width=.7pt] (-1.6625000,1.3875000) -- (-1.0000000,-0.6000000);
\draw[avgray,dash pattern=on 2.4pt off 2.4pt,line width=.7pt] (-1.6625000,1.3875000) -- (0.6000000,0.6333333);
\draw[avgray,dash pattern=on 2.4pt off 2.4pt,line width=.7pt] (-2.8000000,-2.8000000) -- (-1.2000000,-1.2000000);
\draw[avgray,dash pattern=on 2.4pt off 2.4pt,line width=.7pt] (-1.6625000,1.3875000) -- (-2.8750000,2.6000000);
\draw[avgray,dash pattern=on 2.4pt off 2.4pt,line width=.7pt] (2.8000000,-2.8000000) -- (1.2000000,-1.2000000);
\draw[avgray,dash pattern=on 2.4pt off 2.4pt,line width=.7pt] (2.8750000,2.6000000) -- (1.0625000,0.7875000);
\draw[black!80,line width=.85pt] (0.0000000,-0.7666667) -- (0.0000000,-0.6500000);
\draw[avblue,line width=1.15pt] (0.0000000,-0.6500000) -- (-0.9833333,-0.6500000);
\draw[avblue,line width=1.15pt] (-0.9833333,-0.6500000) -- (-1.1500000,-1.1500000);
\draw[avblue,line width=1.15pt] (0.0000000,-0.7666667) -- (-1.1500000,-1.1500000);
\draw[avblue,line width=1.15pt] (0.7833333,-0.0500000) -- (0.6000000,-0.0500000);
\draw[avblue,line width=1.15pt] (0.6000000,-0.0500000) -- (0.0000000,-0.6500000);
\draw[avblue,line width=1.15pt] (1.1500000,-1.1500000) -- (0.7833333,-0.0500000);
\draw[avblue,line width=1.15pt] (1.1500000,-1.1500000) -- (0.0000000,-0.7666667);
\draw[black!80,line width=.85pt] (0.6000000,0.6333333) -- (0.6000000,-0.0500000);
\draw[black!80,line width=.85pt] (0.7833333,-0.0500000) -- (1.0625000,0.7875000);
\draw[black!80,line width=.85pt] (0.6000000,0.6333333) -- (1.0625000,0.7875000);
\draw[black!80,line width=.85pt] (-1.6625000,1.3875000) -- (-0.9833333,-0.6500000);
\draw[black!80,line width=.85pt] (-1.6625000,1.3875000) -- (0.6000000,0.6333333);
\draw[black!80,line width=.85pt] (-2.8000000,-2.8000000) -- (-1.1500000,-1.1500000);
\draw[black!80,line width=.85pt] (-1.6625000,1.3875000) -- (-2.8750000,2.6000000);
\draw[black!80,line width=.85pt] (2.8000000,-2.8000000) -- (1.1500000,-1.1500000);
\draw[black!80,line width=.85pt] (2.8750000,2.6000000) -- (1.0625000,0.7875000);
\end{scope}
\path (-0.4000000,-0.7700000) node[text=avred!70!black,font=\fontsize{10}{11}\selectfont,inner sep=0pt] {$F_0$};
\path (0.6200000,-0.6200000) node[text=avred!70!black,font=\fontsize{10}{11}\selectfont,inner sep=0pt] {$F_1$};
\path (0.8000000,0.3900000) node[text=black!50,font=\fontsize{10}{11}\selectfont,inner sep=0pt] {$F_2$};
\path (-0.7500000,0.2800000) node[text=black!50,font=\fontsize{10}{11}\selectfont,inner sep=0pt] {$F_3$};
\filldraw[fill=avgreen,draw=white,line width=.3pt] (0.0000000,-0.7000000) circle[radius=1.15pt];
\path (0.1800000,-0.9600000) node[fill=white,fill opacity=.85,text opacity=1,text=avgreen,font=\fontsize{12}{13}\selectfont,inner sep=.3pt] {$p_1$};
\filldraw[fill=avgreen,draw=white,line width=.3pt] (0.8000000,0.0000000) circle[radius=1.15pt];
\path (1.1000000,-0.1400000) node[fill=white,fill opacity=.85,text opacity=1,text=avgreen,font=\fontsize{12}{13}\selectfont,inner sep=.3pt] {$p_2$};
\filldraw[fill=avgreen,draw=white,line width=.3pt] (-0.2000000,0.9000000) circle[radius=1.15pt];
\path (-0.4900000,1.1600000) node[fill=white,fill opacity=.85,text opacity=1,text=avgreen,font=\fontsize{12}{13}\selectfont,inner sep=.3pt] {$p_3$};
\filldraw[fill=white,draw=avblue,line width=.8pt] (-0.8000000,-0.6500000) circle[radius=2.3pt];
\filldraw[fill=avblue,draw=white,line width=.3pt] (-0.8000000,-0.6500000) circle[radius=1.15pt];
\path (-1.0900000,-0.9400000) node[fill=white,fill opacity=.85,text opacity=1,text=avblue,font=\fontsize{12}{13}\selectfont,inner sep=.3pt] {$p_4$};
\end{tikzpicture}}
  \caption{$p_4$ reaches the curve;\\ $\Delta_3=0.1$.}
\end{subfigure}

\medskip
{\small
\renewcommand{\arraystretch}{1.12}
\setlength{\tabcolsep}{5pt}
\begin{tabular}{lcccc}
\hline
\multicolumn{5}{c}{$h^{(k)}(x,y)=\min_{0\leq i\leq7}
  \bigl(a_i^{(k)}+q_i\cdot(x,y)\bigr)$}\\
\hline
Slope term $q_i\cdot(x,y)$ & Initial & Step 1 & Step 2 & Step 3\\
 & $a_i^{(0)}$ & $a_i^{(1)}$ & $a_i^{(2)}$ & $a_i^{(3)}$\\
\hline
$i=0:\quad x+y$  & 0   & 1   & 2.6 & 2.7\\
$i=1:\quad -x+y$ & 0   & 1   & 2.6 & 2.7\\
$i=2:\quad -x-y$ & 0   & 1   & 2.6 & 2.6\\
$i=3:\quad x-y$  & 0.4 & 1.4 & 1.4 & 1.4\\
$i=4:\quad 4x$   & 5   & 5   & 5   & 5\\
$i=5:\quad -4x$  & 5   & 5   & 5   & 5\\
$i=6:\quad 4y$   & 5   & 5   & 5   & 5\\
$i=7:\quad -4y$  & 3.9 & 3.9 & 3.9 & 3.9\\
\hline
Moving indices $S_k$ & --- & $\{0,1,2,3\}$ & $\{0,1,2\}$ & $\{0,1\}$\\
Additional increment $\Delta_k$ & --- & $1$ & $1.6$ & $0.1$\\
First event at & --- & $p_3$ & $p_2$ & $p_4$\\
\hline
Gap $\delta_k(p_1)$ & --- & $2.9$ & $1.8$ & $0.2$\\
Gap $\delta_k(p_2)$ & --- & $2.6$ & $1.6$ & ---\\
Gap $\delta_k(p_3)$ & --- & $1$ & --- & ---\\
Gap $\delta_k(p_4)$ & --- & $3.25$ & $1.7$ & $0.1$\\
\hline
\multicolumn{5}{c}{$p_1=(0,-0.7),\qquad p_2=(0.8,0)$}\\
\multicolumn{5}{c}{$p_3=(-0.2,0.9),\qquad p_4=(-0.8,-0.65)$}\\
\hline
\end{tabular}\par}

\smallskip
\begin{minipage}{0.97\linewidth}
\footnotesize
All decimals are exact rational numbers.  The gap $\delta_k(p)$ is measured
before step $k$ against monomials outside $S_k$; released points are omitted.
Each $\Delta_k$ is an additional coefficient increment, not a Euclidean
displacement of an edge.
\end{minipage}

\caption{Three avalanche updates after dropping $p_4$, with
$p_1,p_2,p_3$ initially on the corner locus.  The face $F_i$ corresponds to
the monomial of index $i$.  In each panel, transparent red marks the faces
whose coefficients are raised, dashed gray shows the previous curve, and
blue shows edges with changed supporting lines; black edges retain their
supporting lines.  All panels use the same scale.
(a) Raising $F_0,F_1,F_2,F_3$ by $1$ brings an outside monomial to $p_3$,
allowing $F_3$ to stop moving.
(b) Raising $F_0,F_1,F_2$ by $1.6$ similarly releases $p_2$ and stops $F_2$;
$p_1$ still forces $F_0,F_1$ to move together and $p_4$ is off the curve.
(c) Raising $F_0,F_1$ by $0.1$ brings the $F_0$--$F_3$ edge to $p_4$.
All four points are now on the corner locus, so the configuration is stable.}
\label{fig_avalanche_three_steps}
\end{figure}
\clearpage

In Figure~\ref{fig_avalanche_three_steps}, a point is released when it no
longer forces the moving block to propagate through its tied monomials.
At the first event, the minima at $p_3$ have indices $2,3,7$; keeping
$a_3,a_7$ fixed preserves a tie there while $a_2$ continues to rise.
At the second event, the minima at $p_2$ have indices $1,2,5$; keeping
$a_2,a_5$ fixed similarly allows $a_1$ to continue rising.
These choices resolve the higher-order ties without moving all their
monomials.  At the final configuration the minimal-index sets at
$p_1,p_2,p_3,p_4$ are respectively $\{0,1\}$, $\{2,5\}$,
$\{3,7\}$, and $\{0,3\}$.

\begin{proposition}[Certification of a binary-tie avalanche]
\label{prop_avalanche_certificate}
Assume that every set $M(p)$ used in the closure, except the seed singleton,
has two elements.  Let $b=(b_q)_{q\in\Lambda}$ be any coefficient vector
with $b_q\geq a_q$ such that all currently protected points are corners for
the tropical polynomial defined by $b$.  Then
\[
  a_q+\delta_*\leq b_q\qquad(q\in S).
\]
Consequently, if the current coefficient vector is bounded above by the
coefficients of $G_Pf$, the avalanche update preserves this bound.  If the
algorithm stops with every point of $P$ on the corner locus, its output is
therefore exactly $G_Pf$.
\end{proposition}

\begin{proof}
Put $d_q=b_q-a_q$ and suppose, to the contrary, that
$d=\min_{q\in S}d_q<\delta_*$.  Let
$T=\{q\in S:d_q=d\}$.  The binary ties used to construct $S$ form a
connected graph rooted at $q_0$.  If $T=S$, then at the seed point $p_0$
the vector $b$ can have a multiple minimum only if a monomial outside $S$
reaches or passes the $q_0$-monomial.  Its current gap is at most
$d_{q_0}=d$, contrary to the definition of $\delta_*$.

If $T\neq S$, choose an edge $\{q,r\}=M(p)$ of the closure graph with
$q\in T$ and $r\notin T$.  The two monomials have the same current value at
$p$, but after applying the increments $d_q,d_r$ the $r$-monomial is
strictly higher.  Since $p$ is a corner for $b$, another monomial must reach
or pass the $q$-monomial.  It cannot lie in $S$: a monomial in $T$ with no
current gap would already belong to the binary set $M(p)$, and every other
monomial in $S$ receives a larger increment.  Hence an outside competitor
has current gap at most $d$, again contradicting
$\delta_*>d$.  Thus $\delta_*\leq d_q$ for every $q\in S$.

If $\Lambda$ contains all slopes active at $P$ for $G_Pf$, taking $b$ to
be the restriction of its canonical coefficient vector proves the invariant.
At termination the output lies above $f$, has all required corners, and is
bounded above by the least such series $G_Pf$; equality follows.
\end{proof}

After the update, recompute $M(p)$ and the short competitor cache only at
points in $\bigcup_{q\in S}W_q$, append the corresponding watcher entries, and
repeat.  A formally guaranteed implementation finishes with fair exact
$G_p$ sweeps; Theorem~\ref{thm_finite_stabilization} shows that this fallback
terminates.  In floating-point implementations, equalities and positive gaps
are tested with one consistently chosen tolerance, and a final full scan
serves as the stopping certificate.

\subsubsection*{Why this is faster}

The naive algorithm discovers a chain
\[
  p_0,p_1,p_2,\ldots
\]
one point at a time: applying $G_{p_0}$ may make $p_1$ leave the corner locus,
then applying $G_{p_1}$ may make $p_2$ leave the corner locus, and so on.  The
watcher algorithm instead finds the entire forced component first and raises
all its coefficients by the minimal common increment.

Thus the expensive operation---scanning slopes to find the lowest monomials
and the next outside competitor---is performed only at points whose recorded
data may actually have changed.  Points whose recorded monomials are untouched
are not scanned.  In practice this replaces repeated global recomputation by
a sequence of local, event-driven avalanche updates.

\section{Tropical series}
\label{sec_series}
Let $U\subset\RR^n$.  A \emph{tropical Laurent polynomial}, abbreviated
below to \emph{tropical polynomial}, is a function
\begin{equation}
\label{eq_Ftroppoli}
  f(z)=\min_{q\in\A}(q\cdot z+a_q),
  \qquad z\in U,
\end{equation}
where $\A\subset\ZZ^n$ is finite and $a_q\in\RR$.  The affine function
$q\cdot z+a_q$ is the tropical monomial of slope $q$, and $a_q$ is its
coefficient.  The locus in $U^\circ$ where $f$ is not differentiable is the
\emph{tropical hypersurface} $C(f)$; see \cite{mikh2}.

The minimum of affine functions is concave.  Thus, when $U$ is convex, the
hypograph of $f$ is convex.  The tropical hypersurface $C(f)$ consists of
the points $z\in U^\circ$ for which at least two distinct slopes
$q_1,q_2\in\A$ attain the minimum:
\[
  f(z)=q_1\cdot z+a_{q_1}=q_2\cdot z+a_{q_2}.
\]

\begin{definition}
\label{def_tropseries}
Let $U\subset\RR^n$ have nonempty interior.  A continuous function
$f:U\to\RR$ is a \emph{tropical series} if every $z_0\in U^\circ$ has a
neighborhood $W\subset U$ on which $f|_W$ is a tropical polynomial.
\end{definition}

\begin{definition}[Cf. Definition~\ref{def_tropicalanalytic}]
A \emph{tropical analytic hypersurface} in $U$ is the corner locus
$C(f)\subset U^\circ$ of a tropical series $f$, equivalently, its locus of
nondifferentiability.
\end{definition}

\begin{example}
Tropical $\Theta$-divisors are examples of tropical analytic curves in
$\RR^2$; see \cite{mikhalkin2006tropical}.  Another example is the union
of all horizontal and vertical lines through lattice points:
\[
 C=\bigcup_{k\in\ZZ}\{(k,y):y\in\RR\}
   \cup\bigcup_{k\in\ZZ}\{(x,k):x\in\RR\}.
\]
\end{example}

The following example shows that a tropical series on an open domain need
not extend continuously to its boundary.
\begin{example}
On $(0,1]\times[0,1]$, consider
\[
 C=\bigcup_{n\in\NN}\{(1/n,y):y\in[0,1]\}
   \cup\{(x,1/2):x\in(0,1]\}.
\]
Any tropical series $f$ with $C(f)=C$ satisfies
$f(x,y)\to-\infty$ as $x\to0$ and therefore has no continuous extension
to the missing boundary segment.
\end{example}

\begin{question}
For a tropical series on $\Omega^\circ$, what can be said about the subset
of $\partial\Omega$ across which it extends continuously?  Can this set have
infinitely many connected components or fractal structure?
\end{question}

On nonconvex domains, a slope need not have a globally defined coefficient.

\begin{example}
Let $[y]$ denote the floor of $y$ and set $f(x,y)=\min(3,x+[y])$ on
$U=(U_1\cup U_2)^\circ$, where
\begin{align*}
 U_1&=([0,5]\times[0,1])\cup([4,5]\times[1,2]),\\
 U_2&=([0,5]\times[2,3])\cup([4,5]\times[1,2]).
\end{align*}
Then $f|_{U_1^\circ}=\min(3,x)$ and
$f|_{U_2^\circ}=\min(3,x+2)$.  Thus the monomial of slope $(1,0)$ has
coefficient $0$ in one part of $U$ and coefficient $2$ in the other.
\end{example}

For this reason, we henceforth restrict our attention to tropical series on convex domains.
\section{\texorpdfstring{$\Omega$}{Omega}-tropical series}
\label{sec_omegaseries}

\begin{definition}
\label{def_tropicalanalytic}
Let $\Omega\subset\RR^n$ be closed and convex with nonempty interior.  An
\emph{$\Omega$-tropical series} is a function
$f:\Omega\to\RR_{\geq0}$ satisfying $f|_{\partial\Omega}=0$ and
\begin{equation}
\label{eq_omegatrop}
  f(z)=\inf_{q\in\A}(q\cdot z+c_q),
  \qquad c_q\in\RR,
\end{equation}
where $\A\subset\ZZ^n$ need not be finite.  Its corner locus
$C(f)\subset\Omega^\circ$ is an \emph{$\Omega$-tropical analytic
hypersurface}.
\end{definition}

\begin{question}
An $\Omega$-tropical series is the formal tropical analogue of
$f_t(z)=\sum_{q\in\A}t^{c_q}z^q$ for small $t>0$.  Is $\Omega^\circ$ the
limit, under the coordinatewise logarithmic map, of the convergence domains
of $f_t$?  Is the associated tropical analytic hypersurface the limit of
the logarithmic images of $\{f_t=0\}$ as $t\to0$?  These statements hold
locally, but the behavior near $\partial\Omega$ is unclear.
\end{question}

\begin{lemma}
\label{lemma_estimate}
Let $U\subset\RR^n$ be open, let $K\subset U$ be compact, and let $A>0$.
There are only finitely many $q\in\ZZ^n$ for which an $a_q\in\RR$ and a
$z_0\in K$ exist such that
\[
  q\cdot z+a_q\geq0\quad(z\in U),
  \qquad q\cdot z_0+a_q\leq A.
\]
\end{lemma}
\begin{proof}
Set $R=\operatorname{dist}(K,\RR^n\setminus U)>0$ and $r=R/2$.
For $q\neq0$ and $z_0\in K$, the point
$z_0-rq/|q|$ belongs to $U$.  Therefore
\[
  0\leq q\cdot\left(z_0-r\frac q{|q|}\right)+a_q
  =q\cdot z_0+a_q-r|q|.
\]
If the value at $z_0$ is at most $A$, then $|q|\leq A/r$.  A bounded
subset of $\RR^n$ contains only finitely many lattice points.
\end{proof}
Lemma~\ref{lemma_estimate} supplies the local finiteness needed to replace
``$\inf$'' by ``$\min$'' at interior points.
\begin{lemma}
\label{lemma_welldefinedseries}
At every point $z\in\Omega^\circ$ we have
\[
  \inf_{q\in\A}(q\cdot z+c_q)=\min_{q\in\A}(q\cdot z+c_q).
\]
\end{lemma}
\begin{proof}
\label{proof_welldefinedseries}
If the infimum were not attained at $z_0$, infinitely many distinct
monomials would have value below some fixed number $A>f(z_0)$.  Every
monomial in the presentation is nonnegative on $\Omega$, because their
minimum is nonnegative.  Lemma~\ref{lemma_estimate}, applied to a compact
neighborhood of $z_0$ in $\Omega^\circ$, gives a contradiction.
\end{proof}
At a boundary point admitting no supporting hyperplane with rational normal,
the infimum need not be attained; compare Lemma~\ref{lemma_lomegaiszero}.

Applying Lemma~\ref{lemma_estimate} to compact neighborhoods gives the
following result.

\begin{corollary}

An $\Omega$-tropical series (Definition~\ref{def_tropicalanalytic}) is a
tropical series in the sense of Definition~\ref{def_tropseries}.
\end{corollary}

\begin{lemma}
\label{lemma_usualisomegatropical}
Suppose that $\Omega$ is convex and that $f:\Omega\to\RR_{\geq0}$ is
continuous.  If $f|_{\Omega^\circ}$ is a tropical series and
$f|_{\partial\Omega}=0$, then $f$ is an $\Omega$-tropical series.
\end{lemma}

\begin{proof}
Local tropical polynomials are concave, and local concavity on the convex
set $\Omega^\circ$ implies concavity.  If an affine function
$L(z)=q\cdot z+c_q$ agrees with $f$ on a nonempty open subset, concavity
implies that $L\geq f$ throughout $\Omega^\circ$, and hence on $\Omega$ by
continuity.  Let $g$ be the infimum of all such affine pieces.  Locally the
pieces defining a tropical-polynomial presentation occur in this collection,
so $g=f$ on $\Omega^\circ$.

It remains to check the boundary.  Fix $x\in\partial\Omega$ and
$y\in\Omega^\circ$.  For sufficiently small $t>0$, put
$x_t=(1-t)x+ty$ and $y_t=(1-2t)x+2ty$; both points lie in
$\Omega^\circ$.  Choose an affine piece $L_t$ active at $x_t$.  Since
$x_t=(x+y_t)/2$ and $L_t\geq f\geq0$ on $\Omega$, we have
\[
  L_t(x)=2L_t(x_t)-L_t(y_t)\leq2f(x_t).
\]
Consequently $0\leq g(x)\leq2f(x_t)\to0$.  Thus $g=f$ also on
$\partial\Omega$, giving the required global presentation.
\end{proof}

\section{The tropical distance function}
\label{sec_distance}
Every domain admits the zero tropical series, whose corner locus is empty.
Unbounded convex domains need not admit a non-constant $\Omega$-tropical series; examples
include $\RR^n$ and a half-space with irrational boundary normal.

\begin{definition}
\label{def_singlesupport}
Let $\Omega\subset\RR^n$.  For $q\in\ZZ^n$, put
\[
  c_q=\inf_{z\in\Omega}q\cdot z\in\RR\cup\{-\infty\}
\]
and let $\A_\Omega$ be the set of $q$ for which $c_q$ is finite.  Define
\begin{equation}
\label{eq_lij}
  l^q_\Omega(z)=q\cdot z-c_q.
\end{equation}
If $\Omega$ is bounded, then $\A_\Omega=\ZZ^n$.
\end{definition}

For $q\neq0$, the function $l^q_\Omega$ is strictly positive on
$\Omega^\circ$.  Also, $\A_\Omega$ always contains $0\in\ZZ^n$.  To have a
nonconstant $\Omega$-tropical series we must have
$\A_\Omega\neq\{0\}$.  If $\Omega$ is compact, then
$\A_\Omega=\ZZ^n$.

From now on, $\Omega$ is compact and convex with nonempty interior.

\begin{definition}
\label{def_generalweighteddistance}
The \emph{tropical distance function} on $\Omega$ is
\[
  l_\Omega(z)=\inf_{q\in\ZZ^n\setminus\{0\}}l^q_\Omega(z).
\]
\end{definition}

\begin{remark}
\label{rem_lomegaestimate}
If $L(z)=q\cdot z+c$ is nonnegative on $\Omega$ and
$q\in\ZZ^n\setminus\{0\}$, then $L\geq l_\Omega$ on $\Omega$.
\end{remark}

\begin{lemma}
The function $l_\Omega$ is a tropical series on $\Omega^\circ$.
\end{lemma}

\begin{proof}
Fix a compact set $K\Subset\Omega^\circ$ and one nonzero slope $q_0$.
On $K$, the infimum defining $l_\Omega$ is bounded above by
$A=\max_K l_\Omega^{q_0}$.  Every term $l_\Omega^q$ is nonnegative on
$\Omega$.  Lemma~\ref{lemma_estimate} therefore leaves only finitely many
slopes that can attain the infimum on $K$.  Hence $l_\Omega|_K$ is a
tropical polynomial.
\end{proof}

\begin{lemma}
\label{lemma_lomegaiszero}
If $\Omega$ is compact and convex with nonempty interior, then $l_\Omega$
is an $\Omega$-tropical series.
\end{lemma}

\begin{proof}
\label{proof_lomegaiszero}
It remains to prove that $l_\Omega$ tends to zero at every boundary point.
Fix $x\in\partial\Omega$ and choose a nonzero supporting normal
$\alpha\in\RR^n$, oriented so that
\[
  \alpha\cdot(z-x)\geq0\qquad(z\in\Omega).
\]
Dirichlet's simultaneous approximation theorem gives positive integers
$r_i\to\infty$ and vectors $q_i\in\ZZ^n\setminus\{0\}$ such that
\[
  e_i=q_i-r_i\alpha\longrightarrow0.
\]
Writing $D=\operatorname{diam}(\Omega)$, we obtain
\begin{align*}
 l_\Omega^{q_i}(x)
 &=q_i\cdot x-\inf_{z\in\Omega}q_i\cdot z
   =\sup_{z\in\Omega}q_i\cdot(x-z)\\
 &\leq \sup_{z\in\Omega}e_i\cdot(x-z)
 \leq D|e_i|\longrightarrow0.
\end{align*}
Hence $l_\Omega(x)=0$.  Given $\varepsilon>0$, choose $i$ with
$l_\Omega^{q_i}(x)<\varepsilon/2$.  The affine function
$l_\Omega^{q_i}$ is continuous, so
$l_\Omega(z)\leq l_\Omega^{q_i}(z)<\varepsilon$ for all
$z\in\Omega$ sufficiently close to $x$.  Thus $l_\Omega$ extends
continuously by zero to $\partial\Omega$, and
Lemma~\ref{lemma_usualisomegatropical} completes the proof.
\end{proof}

The tropical distance is the pointwise minimum of all nonconstant integral
affine functions that are nonnegative on $\Omega$.  In particular, it gives
a canonical nontrivial $\Omega$-tropical series for every compact convex
domain with nonempty interior.  For unbounded convex domains the set of
admissible slopes and the boundary behavior require separate hypotheses; see
\cite{us_series} for the planar discussion.

\section{Shrinking operators \texorpdfstring{$G_\p$}{Gp}}
\label{sec_gp}

If $f$ is a nonzero $\Omega$-tropical series, then $C(f)$ is nonempty.
The connected components of $\Omega^\circ\setminus C(f)$ are convex; we
call them the \emph{faces} of $f$.  Fix a finite set
$P=\{\p_1,\dots,\p_m\}\subset\Omega^\circ$.
\begin{definition}
\label{def_vOmegaH}
Let $V(\Omega,P,f)$ be the set of $\Omega$-tropical series $g$ such that
$g\geq f$ and $P\subset C(g)$.
\end{definition}
\begin{lemma}
\label{lemma_Visnonempty}
The set $V(\Omega,P,f)$ is nonempty.
\end{lemma}
\begin{proof}
For $p\in P$, set
\[
  h_p(z)=\min\{l_\Omega(z),l_\Omega(p)\}.
\]
At $p$, a nonconstant monomial attaining $l_\Omega(p)$ ties the constant
monomial, so $h_p$ is nonsmooth there.  Sums of tropical series are
tropical series: one distributes the two minima, and the resulting slopes
are still integral.  Moreover, the superdifferential of a sum of concave
functions is the Minkowski sum of their superdifferentials, so adding another
concave function cannot remove the corner of $h_p$ at $p$.  Hence
\[
  f+\sum_{p\in P}h_p\in V(\Omega,P,f).
\]
\end{proof}
Clearly, if $f\geq g$ then $V(\Omega,P,f)\subset V(\Omega,P,g)$.

\begin{definition}
\label{def_gp}
For finite $P\subset\Omega^\circ$, define
\[
  G_Pf(z)=\inf\{g(z):g\in V(\Omega,P,f)\}.
\]
For $P=\{\p\}$, write $G_\p=G_{\{\p\}}$.
\end{definition}

We call the $G_p$ \emph{shrinking operators} because they shrink the face
containing $p$.  In \cite{us_series} their two-dimensional versions were
called wave operators because of their relation to wave dynamics in the
sandpile model; see \cite{announce,us}.
  
\begin{lemma}\label{lem_gpmponotone}
If $g\leq f$ are $\Omega$-tropical series, then $G_Pg\leq G_Pf$.
\end{lemma}

\begin{proof}
Since $f\geq g$, we have
$V(\Omega,P,f)\subset V(\Omega,P,g)$.  Taking pointwise infima proves the
claim.
\end{proof}

An $\Omega$-tropical series may have many presentations as a minimum of
affine functions.  For example, on $[0,1]^2$ the functions
$\min(x,1-x,y,1-y,1/3)$ and
$\min(x,1-x,y,1-y,1/3,2x,5-2x)$ agree.
\begin{definition}
\label{def_canonicalseries}
Following \cite[Lemma~5.3]{kiselman2014questions}, for $q\in\ZZ^n$ define
\[
  a_q(f)=\sup_{z\in\Omega}\bigl(f(z)-q\cdot z\bigr).
\]
Then
\begin{equation}
\label{eq_series}
  f(z)=\inf_{q\in\ZZ^n}\bigl(q\cdot z+a_q(f)\bigr).
\end{equation}
We call \eqref{eq_series} the \emph{canonical form} of $f$.  The coefficient
$a_q(f)$ is the smallest coefficient for which the corresponding affine
function majorizes $f$ on $\Omega$; hence the canonical form is unique.
\end{definition}

Indeed, every canonical affine function majorizes $f$.  Conversely, if
$f=\inf_q(q\cdot z+c_q)$ is any presentation, then
$a_q(f)\leq c_q$.  Taking the infimum proves \eqref{eq_series}.  In the
interior, the infimum is a locally finite minimum by
Lemma~\ref{lemma_estimate}.

\begin{example}
\label{ex_bigform}
For $f(x,y)=\min(x,1-x,y,1-y,1/3)$ on $\Omega=[0,1]^2$, the
canonical coefficients are
\[
  a_{00}(f)=\frac13,
  \qquad
  a_{ij}(f)=-\min_{(x,y)\in\Omega}(ix+jy)
  \qquad ((i,j)\neq(0,0)).
\]
\end{example}
\begin{proof}
For $(i,j)\neq(0,0)$, the displayed value is the smallest coefficient for
which $ix+jy+a_{ij}$ is nonnegative on $\Omega$, and hence it majorizes
$f$.  The smallest constant majorizing $f$ is
$\max_\Omega f=1/3$.
\end{proof}

\begin{figure}[htbp]
\centering
\begin{subfigure}[t]{0.49\linewidth}
  \centering
  \includegraphics[width=\linewidth]{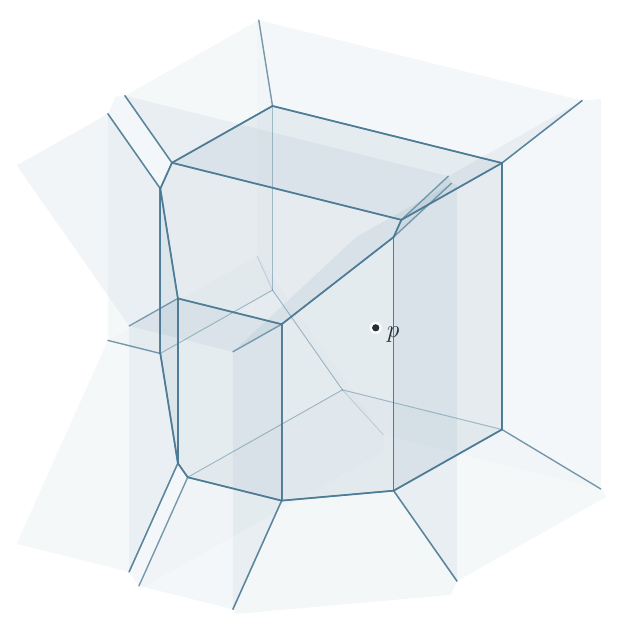}
  \caption{Before applying $G_p$.}
\end{subfigure}\hfill
\begin{subfigure}[t]{0.49\linewidth}
  \centering
  \includegraphics[width=\linewidth]{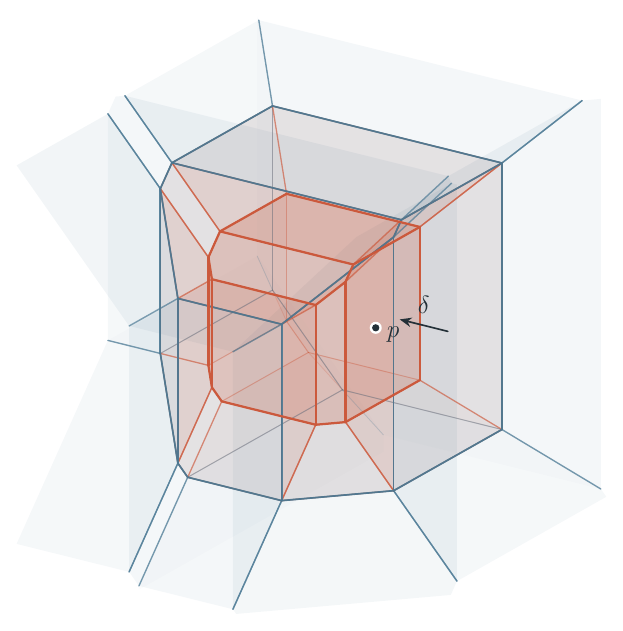}
  \caption{The action of $G_p$.}
\end{subfigure}
\caption{The shrinking operator $G_p$ on a tropical surface in $\RR^3$.
Left: $p$ lies inside a three-dimensional chamber of the complement of
$C(f)$.  Right: raising the coefficient of the monomial uniquely minimal
at $p$ by $\delta$ contracts this chamber until its boundary reaches $p$.
The new portions of $C(G_pf)$ are red: the inner chamber walls and the
continuations of adjacent facets connecting them to the unchanged exterior
sheets.  The original surface is shown in transparent blue; its former
chamber walls are retained only for comparison.}
\label{fig_3dpicture}
\end{figure}

Lemma~\ref{lem_singlegp} shows that $G_\p$ contracts the face containing
$\p$ until its boundary reaches $\p$; see Figures~\ref{fig_3dpicture}
and~\ref{fig_ShrinkPhi}.
Proposition~\ref{prop_gpsconvergence} then constructs $G_P$ by repeated
one-point operations.

We denote by $0_\Omega$ the function $f\equiv 0$ on $\Omega$.

\begin{lemma}
For $\p\in\Omega^\circ$,
\[
  (G_\p0_\Omega)(z)=\min\{l_\Omega(z),l_\Omega(\p)\}.
\]
\end{lemma}
\begin{proof}
The displayed function is nonnegative, vanishes on $\partial\Omega$, and is
nonsmooth at $\p$: the constant monomial ties a nonconstant monomial
attaining $l_\Omega(\p)$.  It therefore belongs to
$V(\Omega,\{\p\},0_\Omega)$.  Conversely, let
$h\in V(\Omega,\{\p\},0_\Omega)$.  Every nonconstant canonical monomial of
$h$ is at least $l_\Omega$, by Remark~\ref{rem_lomegaestimate}.  If the
constant coefficient were smaller than $l_\Omega(\p)$, it would be the
unique minimum at $\p$, contradicting $\p\in C(h)$.  Consequently
\[
  h\geq\min\{l_\Omega,l_\Omega(\p)\}.
\]
\end{proof}

\begin{proposition} 
\label{prop_upperbound}
Let $P=\{\p_1,\dots,\p_m\}\subset\Omega^\circ$.  For every
$\Omega$-tropical series $f$,
\[
 f\leq G_Pf
 \leq f+\sum_{i=1}^m\min\{l_\Omega,l_\Omega(\p_i)\}
 \leq f+m\,l_\Omega.
\]
\end{proposition}
\begin{proof}
The first inequality is part of the definition.  The middle expression is
the element of $V(\Omega,P,f)$ constructed in
Lemma~\ref{lemma_Visnonempty}; the last inequality is immediate.
\end{proof}

\begin{lemma}
\label{lemma_tropicalseries}
The operator $G_P$ maps $\Omega$-tropical series to $\Omega$-tropical series. 
\end{lemma}
\begin{proof}
\label{proof_tropicalseries}
Choose once and for all an upper candidate
$h_0\in V(\Omega,P,f)$ as in Lemma~\ref{lemma_Visnonempty}.  For
$q\in\ZZ^n$ put
\[
  a_q=\inf_{h\in V(\Omega,P,f)}a_q(h),
  \qquad
  H(z)=\inf_{q\in\ZZ^n}(q\cdot z+a_q).
\]
Infima commute, so
\[
 H(z)=\inf_q\inf_h(q\cdot z+a_q(h))
     =\inf_h\inf_q(q\cdot z+a_q(h))
     =\inf_{h\in V(\Omega,P,f)}h(z)=G_Pf(z).
\]
Moreover, $a_q\geq a_q(f)$, so every affine term defining $H$ is
nonnegative on $\Omega$.  We also have $f\leq H\leq h_0$.

Let $K\Subset\Omega^\circ$ be compact and take
$A>\max_Kh_0$.  Lemma~\ref{lemma_estimate} leaves only finitely many slopes
whose affine terms can have value below $A$ anywhere on $K$.  At every
point of $K$ at least one such term has value below $A$, because
$H\leq h_0$.  Hence the infimum defining $H|_K$ is the minimum of this
fixed finite family.  The inequalities
$f\leq H\leq h_0$ show at once that $H$ extends continuously by zero to
$\partial\Omega$.

It remains to prove that $H$ is nonsmooth at every $p\in P$.  Suppose that
$H$ has a unique minimizing slope $q_0$ at such a point.  By
Lemma~\ref{lemma_estimate}, only finitely many affine terms have value at
most $H(p)+1$ at $p$.  The minimum of their positive gaps from $H(p)$,
together with $1$, gives a number $\gamma>0$ such that
\[
  q\cdot p+a_q\geq H(p)+\gamma\qquad(q\neq q_0).
\]
The set
\[
  M=\{q:q\cdot p+a_q\leq H(p)+2\gamma\}
\]
is finite.  For every $q\in M$, choose $h_q\in V(\Omega,P,f)$ with
$a_q(h_q)<a_q+\gamma/4$, and put $h=\min_{q\in M}h_q$.
This finite minimum is again an $\Omega$-tropical series above $f$.  At any
$p'\in P$, choose an $h_q$ with minimal value at $p'$; two distinct
monomials of that $h_q$ attain the minimum there, so they also certify that
$h$ is nonsmooth at $p'$.  Thus $h\in V(\Omega,P,f)$.

Because $h\in V(\Omega,P,f)$, its canonical coefficients satisfy
$a_r(h)\geq a_r$ for all $r$.  On the other hand,
$a_r(h)\leq a_r(h_r)<a_r+\gamma/4$ for $r\in M$.  It follows that the
$q_0$-monomial is the unique minimum of $h$ at $p$, a contradiction.
Therefore $P\subset C(H)$.  We have proved that $H=G_Pf$ is an
$\Omega$-tropical series.
\end{proof}

\begin{remark}
\label{rem_cut}
Let $f=G_P0_\Omega$ and suppose that
$0<\e<\min_{\p\in P}f(\p)$.  Then
\[
  G_P0_\Omega=G_P\min\{f,\e\}.
\]
\end{remark}
\begin{proof}
Monotonicity gives
\[
  G_P0_\Omega\leq G_P\min\{f,\e\}.
\]
Conversely, $f$ majorizes $\min\{f,\e\}$ and is nonsmooth at every point
of $P$.  Hence
\[
  G_P\min\{f,\e\}\leq f.
\]
\end{proof}

\section{Flow version of the operators \texorpdfstring{$G_\p$}{Gp}}
\label{sec_flow}

The following operation raises one canonical coefficient and leaves all
others unchanged.

\begin{definition}
\label{def_add}
For an $\Omega$-tropical series $f$ in canonical form, $c\geq0$, and
$q\in\ZZ^n$, set
\[
 (\Add_q^cf)(z)=
 \min\left\{a_q(f)+c+q\cdot z,\,
       \inf_{\substack{r\in\ZZ^n\\r\neq q}}
       \bigl(a_r(f)+r\cdot z\bigr)\right\}.
\]
We use this notation only while the raised monomial remains canonical; this
is automatic along the partial shrinking flows below.
\end{definition}

\begin{figure}[ht]
    \centering

\begin{tikzpicture}[scale=0.4]
\draw[very thick](0,0)--++(1,1)--++(0,7)--++(4,0)--++(1,-1)--++(0,-6)--++(-5,0);
\draw[very thick](1,8)--++(-1,1);
\draw[very thick](5,8)--++(0,1);
\draw[very thick](6,7)--++(1,0);
\draw[very thick](6,1)--++(1,-1);
\draw(3,4)node{$\bullet$};
\draw(3,4)node[right]{$\p$};
\draw(4.5,6.5)node{\Large$\Phi$};
\draw[->][very thick](8,4)--(10,4);

\begin{scope}[xshift=300]
\draw[very thick](0,0)--++(2,2)--++(0,5)--++(3,0)--++(0,-5)--++(-3,0);
\draw[very thick](2,7)--++(-2,2);
\draw[very thick](5,7)--++(0,2);
\draw[very thick](5,7)--++(2,0);
\draw[very thick](5,2)--++(2,-2);
\draw(3,4)node{$\bullet$};
\draw(3,4)node[right]{$\p$};
\draw[->][very thick](8,4)--(10,4);
\end{scope}

\begin{scope}[xshift=600]
\draw[very thick](0,0)--++(3,3)--++(0,3)--++(1,0)--++(0,-3)--++(-1,0);
\draw[very thick](3,6)--++(-3,3);
\draw[very thick](5,7)--++(0,2);
\draw[very thick](5,7)--++(2,0);
\draw[very thick](5,7)--++(-1,-1);
\draw[very thick](4,3)--++(3,-3);
\draw(3,4)node{$\bullet$};
\draw(3,4)node[left]{$\p$};
\end{scope}

\end{tikzpicture}
\caption{The flow from Remark~\ref{rem_smooth}.  The operator $G_\p$
shrinks the face $\Phi$ containing $\p$.  The three panels correspond to
$t=0$, $t=1/2$, and $t=1$ in $\Add_q^{ct}f$.  The combinatorial type may
change during the flow; cf.\ \cite[Figure~2]{vafa2012supersymmetric}.}
\label{fig_ShrinkPhi}
\end{figure}

\begin{lemma}
\label{lem_singlegp}
Let $f$ be an $\Omega$-tropical series in canonical form, and suppose that
$\p\in\Omega^\circ\setminus C(f)$.  Let $q_0$ be its unique minimizing
slope at $\p$, and put
\begin{equation}
\label{eq_gp}
  g(z)=\inf_{q\in\ZZ^n\setminus\{q_0\}}
       (q\cdot z+a_q(f)).
\end{equation}
Then
\[
  G_\p f=\Add_{q_0}^{c}f,
  \qquad c=g(\p)-f(\p)>0.
\]
\end{lemma}

\begin{proof}
Put $b=g(\p)-q_0\cdot\p$, so that
$b-a_{q_0}(f)=g(\p)-f(\p)=c$, and define
\[
  H(z)=\min\{g(z),q_0\cdot z+b\}.
\]
The two displayed terms tie at $\p$, and $H\geq f$.  Together with the
boundary verification below, this gives $H\in V(\Omega,\{\p\},f)$ and
$G_\p f\leq H$.

For completeness, $H$ still has zero boundary values.  If the
$q_0$-term is positive at a boundary point, the remaining canonical terms
already have infimum zero there.  If $q_0\neq0$ and its value is zero, the
$2q_0$-term also has canonical value zero there, because twice the
$q_0$-term still majorizes $f$.  Finally, if $q_0=0$ and its coefficient is
zero, then $f=0_\Omega$, and the assertion follows from
Lemma~\ref{lemma_lomegaiszero}.  Thus $H$ has the required zero boundary
values.

Conversely, take $k\in V(\Omega,\{\p\},f)$ and write it in canonical form.
Since $k\geq f$, all its canonical coefficients are at least the
corresponding coefficients of $f$.  If its $q_0$-coefficient were smaller
than $b$, then at $\p$ the $q_0$-monomial would be strictly below every
other monomial, whose minimum is at least $g(\p)$.  This would make $k$
smooth at $\p$, a contradiction.  Thus every canonical monomial of $k$
majorizes the corresponding term in $H$, and $k\geq H$.  Taking the
pointwise infimum over $k$ gives $G_\p f\geq H$.
For $r\neq q_0$, the inequalities $H\geq f$ and
$H\leq r\cdot z+a_r(f)$ imply $a_r(H)=a_r(f)$; the raised
$q_0$-term is active at $\p$.  Thus the displayed coefficients are
canonical.
Finally, Proposition~\ref{prop_upperbound} gives
$f\leq H=G_\p f\leq f+\min\{l_\Omega,l_\Omega(\p)\}$, so the same
sandwich also proves continuity at the boundary.
\end{proof}

If $\p\in C(f)$, then $f\in V(\Omega,\{\p\},f)$ and therefore
$G_\p f=f$.

Each face of $f$ is a domain of linearity.  Thus it has a unique slope
$q\in\ZZ^n$, and $f(z)=q\cdot z+a_q(f)$ throughout that face.

\begin{remark}
\label{rem_smooth}
Suppose that $G_\p f=\Add_q^cf$.  The family
\[
  t\longmapsto\Add_q^{ct}f,\qquad 0\leq t\leq1,
\]
interpolates continuously between $f$ and $G_\p f$.  Geometrically it
shrinks the face containing $\p$ until the boundary of that face reaches
$\p$; see Figure~\ref{fig_ShrinkPhi}.
\end{remark}

\section{Dynamics generated by
\texorpdfstring{$G_\p$ for $\p\in P$}{Gp for p in P}}
\label{sec_dynamic}

Let $P=\{\p_1,\dots,\p_s\}\subset\Omega^\circ$, and let
$\p_1',\p_2',\ldots$ be a fair sequence in $P$.  Starting from an
$\Omega$-tropical series $f$, define
\[
  f_1=f,\qquad f_{m+1}=G_{\p_m'}f_m.
\]

\begin{proposition}\label{prop_gpsconvergence}
The sequence $\{ f_m \}_{m=1}^\infty$ uniformly converges to $G_P f$.
\end{proposition}

\begin{proof}
Write $F=G_Pf$.  Monotonicity and $G_pF=F$ for $p\in P$ give
\[
  f_m\leq f_{m+1}\leq F.
\]
Only finitely many slopes can ever be changed.  Indeed, if $G_pf_m$ changes
the coefficient of $q$, then the $q$-monomial is minimal at $p$ and
\[
  q\cdot p+a_q(f_m)=f_m(p)\leq F(p).
\]
Every canonical monomial is nonnegative on $\Omega$.  Applying
Lemma~\ref{lemma_estimate} to fixed compact neighborhoods of the finitely
many points of $P$, with $A=1+\max_{p\in P}F(p)$, therefore gives a finite set
$S\subset\ZZ^n$ containing every slope that can be changed.

For $q\in S$, the sequence $a_q(f_m)$ is nondecreasing and bounded above
by $a_q(F)$; let its limit be $b_q$.  Coefficients outside $S$ never change.
Define $\widetilde f$ by these limiting coefficients.  If
\[
  \eta_m=\max_{q\in S}\bigl(b_q-a_q(f_m)\bigr),
\]
then the elementary inequality for minima gives
\[
  0\leq\widetilde f-f_m\leq\eta_m\longrightarrow0
  \quad\text{uniformly on }\Omega.
\]
Thus $\widetilde f$ is an $\Omega$-tropical series and
$\widetilde f\leq F$.

We claim that $\widetilde f$ is nonsmooth at every $p\in P$.  Otherwise,
local finiteness would give a unique minimizing slope at $p$ separated from
the second-lowest value by some $\gamma>0$.  For all sufficiently large
$m$, the same slope would be uniquely minimal for $f_m$, with gap at least
$\gamma/2$.  Every subsequent occurrence of $p$ in the fair sequence would
then increase that coefficient by at least $\gamma/2$, contradicting its
convergence.  Hence $\widetilde f\in V(\Omega,P,f)$, so
$F\leq\widetilde f$.  Therefore $\widetilde f=F$.
\end{proof}

\begin{remark}
Theorem~\ref{thm_finite_stabilization} strengthens the preceding proposition:
for an arbitrary compact convex domain and arbitrary interior points, every
fair sequence $\{f_m\}$ stabilizes after finitely many applications.  In the
special case where $\Omega$ is a lattice polytope and the points of $P$ are
lattice points, this also follows directly because all coefficient increments
$c$ in $G_\p=\Add_q^c$ are integers.
\end{remark}

\begin{lemma}
\label{lemma_ecloseseries}
Let $U\subset\RR^n$ be open, let $\B\subset\ZZ^n$ be finite, and write
\[
 f(z)=\min_{q\in\B}(q\cdot z+a_q),\qquad
 g(z)=\min_{q\in\B}(q\cdot z+a_q+\delta_q).
\]
Suppose that $|\delta_q|\leq\e$ for every $q$.  If
$p\in C(f)$ and $\overline B(p,2\e)\subset U$, then
\[
  \operatorname{dist}(p,C(g))\leq2\e.
\]
The same assertion holds with $f$ and $g$ interchanged.  If either
$0\leq\delta_q\leq\e$ for all $q$ or
$-\e\leq\delta_q\leq0$ for all $q$, then $2\e$ can be replaced by
$\e$.
\end{lemma}
\begin{proof}
Suppose that $g$ has no corner in a ball about $p$.  It is then represented
there by one monomial, with slope $q_3$.  Choose a slope $q_1\neq q_3$
active for $f$ at $p$.  For
\[
 D(z)=q_1\cdot z+a_{q_1}+\delta_{q_1}
      -(q_3\cdot z+a_{q_3}+\delta_{q_3}),
\]
minimality of the $q_3$-monomial gives $D>0$ throughout that ball, whereas
minimality of the $q_1$-monomial for $f$ gives
\[
 0<D(p)\leq\delta_{q_1}-\delta_{q_3}\leq2\e.
\]
Moving from $p$ in the direction
$-(q_1-q_3)/|q_1-q_3|$ decreases $D$ at speed
$|q_1-q_3|\geq1$.  Thus $D$ reaches zero within distance $2\e$, which
contradicts the assumed absence of a corner.  Taking slightly larger balls
handles equality.  If all perturbations have the same sign, then
$|\delta_{q_1}-\delta_{q_3}|\leq\e$.  Interchanging $f$ and $g$ proves
the reverse assertion.
\end{proof}

\begin{corollary}
The dynamics from $f$ to $G_Pf$ can be embedded in a continuous path of
$\Omega$-tropical series.
\end{corollary}
\begin{proof}
By Theorem~\ref{thm_finite_stabilization}, only finitely many one-point
updates are nontrivial.  Concatenate their coefficient flows from
Remark~\ref{rem_smooth}.
\end{proof}

\begin{remark}
\label{rem_curvesnearlimit}
On a compact set whose $2\e$-neighborhood lies in $\Omega^\circ$,
Lemma~\ref{lemma_ecloseseries} turns an $\e$-bound on the coefficients into
the corresponding local Hausdorff bound for corner loci.  The qualification
away from $\partial\Omega$ is necessary.
\end{remark}

\begin{definition}
For two $\Omega$-tropical series in canonical form, define their coefficient
distance by
\[
  \rho_\Omega(f,g)=\sup_{q\in\ZZ^n}|a_q(f)-a_q(g)|
  \in[0,\infty].
\]
\end{definition}

\begin{lemma}
\label{lem_distance}
If $f,g$ are two $\Omega$-tropical series and $\p\in\Omega^\circ$, then
\[
  \rho_\Omega(G_\p f,G_\p g)\leq\rho_\Omega(f,g).
\]
\end{lemma}
\begin{proof}
If $\rho_\Omega(f,g)=\infty$, the assertion is immediate, so assume that
this distance is finite.
At $\p$, put
\[
 x_q=a_q(f)+q\cdot\p,\qquad y_q=a_q(g)+q\cdot\p.
\]
The one-point operation on these value vectors is
\[
  (Tx)_q=\max\left\{x_q,\inf_{r\neq q}x_r\right\}.
\]
Indeed, this formula raises the unique lowest entry to the second-lowest
value and changes nothing if the minimum is already multiple.  Both the
infimum and the maximum are $1$-Lipschitz in the supremum norm.  Hence
\[
  \|Tx-Ty\|_\infty\leq\|x-y\|_\infty
  =\rho_\Omega(f,g).
\]
Subtracting $q\cdot\p$ from the $q$th entry gives the canonical
coefficients after applying $G_\p$, by Lemma~\ref{lem_singlegp}.
\end{proof}

\begin{corollary}
\label{cor_gp_nonexpansive}
For finite $P\subset\Omega^\circ$,
\[
  \rho_\Omega(G_Pf,G_Pg)\leq\rho_\Omega(f,g).
\]
In particular,
$\|G_Pf-G_Pg\|_\infty\leq\rho_\Omega(f,g)$.
\end{corollary}

\begin{proof}
Apply the same fair sequence of one-point operators to $f$ and $g$.
Lemma~\ref{lem_distance} gives the coefficient estimate at every step.
Proposition~\ref{prop_gpsconvergence} and
Lemma~\ref{lem_finite_coefficient_lipschitz} allow passage to the limit.
The function estimate follows from the elementary Lipschitz bound for the
minimum of affine functions.
\end{proof}

\section{Mild singularities}
\label{sec_mild}
Let
\[
  f(z)=\min_{q\in\B}(q\cdot z+a_q)
\]
be a tropical polynomial.  We may replace $\B$ by the saturated support
$\operatorname{conv}(\B)\cap\ZZ^n$, assigning to each added slope its
canonical coefficient; this does not change $f$.  The lower faces of the
extended Newton polyhedron
\[
  \widetilde N(f)=
  \operatorname{conv}\{(q,r):q\in\B,\ r\geq a_q\}
  \subset\RR^{n+1}
\]
project to a regular subdivision of
$N(f)=\operatorname{conv}(\B)$.  If $z\in\RR^n$ and
\[
  \B_z=\{q\in\B:q\cdot z+a_q=f(z)\},
\]
then $\operatorname{conv}(\B_z)$ is the cell dual to the face of the
corner locus containing $z$.

\begin{definition}
The tropical polynomial $f$ has \emph{only mild singularities} if every cell
of its lower regular subdivision contains no lattice points other than its
vertices.  More generally, if $h$ is a tropical series on a domain $D$ and
$U\subset D^\circ$, we say that $C(h)$ has only mild singularities
\emph{on $U$} if, at every $z\in C(h)\cap U$, the convex hull of the
canonical slopes active at $z$ contains no lattice points other than its
vertices.  The active set is finite by Lemma~\ref{lemma_estimate}, so this
local definition is well posed.
\end{definition}

\begin{remark}
\label{rem_mild}
Because the support is saturated, mildness is equivalent to saying that
every lattice point of $N(f)$ is a vertex of the lower subdivision.  Equivalently,
each corresponding monomial is uniquely minimal on a nonempty open subset
of $\RR^n$.
\end{remark}

In dimension two, the one-dimensional dual cells are primitive lattice
segments.  The two-dimensional cells are unimodular triangles or empty
lattice parallelograms, corresponding respectively to smooth vertices and
ordinary nodes of the tropical curve.

\begin{definition}
\label{def_qpolygon}
A compact convex set $Q\subset\RR^n$ with nonempty interior is called a
\emph{$\QQ$-polytope} if it is a finite intersection of half-spaces whose
normal vectors lie in $\ZZ^n$.  Thus its facets have rational directions;
the constant terms of their equations need not be rational.
\end{definition}

\begin{lemma}[Local preservation under a partial update]
\label{lemma_mild}
Let $Q$ be a $\QQ$-polytope, let $f$ be a $Q$-tropical series, and let
$p\in U\subset Q^\circ$ with $p\notin C(f)$.  Suppose that
$G_pf=\Add_q^cf$ and that a finite set $\B\subset\ZZ^n$ has the following
properties:
\begin{enumerate}[(i)]
  \item $\B=\operatorname{conv}(\B)\cap\ZZ^n$;
  \item every $r\in\B$ is uniquely minimal on a nonempty open subset of
  $Q^\circ$ for $f$;
  \item every canonical slope that is active at a point of $U$ for some
  $\Add_q^{ct}f$, $0\leq t\leq1$, belongs to $\B$.
\end{enumerate}
Then $\Add_q^{ct}f$ has only mild singularities on $U$ for every
$0\leq t<1$.  Moreover, every slope in $\B$ remains uniquely minimal
somewhere in $Q^\circ$ throughout this half-open flow.
\end{lemma}

\begin{proof}
Condition~(iii) at $t=0$ implies $q\in\B$.  Raising the $q$-coefficient
cannot destroy a region where a different slope $r\in\B$ is uniquely
minimal.  For $t<1$, the $q$-monomial itself remains strictly minimal at
$p$, because $c$ is exactly the first value at which it ties another
monomial there.  Thus every slope in $\B$ remains a lower vertex.

Fix $z\in U$ and let $A_z\subset\B$ be the set of slopes active there.
If a lattice point $r$ lies in $\operatorname{conv}(A_z)$, then
$r\in\B$ by condition~(i).  Since $r$ is a lower vertex, it cannot lie in
this cell of the regular subdivision without being one of its vertices.
Therefore $\operatorname{conv}(A_z)$ contains no other lattice points,
which is precisely mildness on $U$.
\end{proof}

\begin{remark}
At $t=1$ the $q$-monomial may cease to be a lower vertex, and mildness can
fail.  This endpoint phenomenon is why the approximation theorem below uses
adaptive partial updates rather than subtracting one fixed number from a
precomputed list of coefficient increments.  Condition~(iii) in
Lemma~\ref{lemma_mild} is also essential: close to $\partial Q$, a partial
update can reveal canonical monomials absent from the original small
canonical form.
\end{remark}

\section{Polyhedral exhaustion}
\label{sec_exhausting}

When the domain must be displayed, we write $G_P^D$ for the shrinking
operator on a domain $D$.

\begin{definition}
\label{def_minimalform}
If only finitely many canonical monomials of a tropical series $f$ are
minimal somewhere in the interior of its domain, their minimum is called the
\emph{small canonical form} of $f$.
\end{definition}

For example, the small canonical form in Example~\ref{ex_bigform} is
$\min(x,y,1-x,1-y,1/3)$.

\begin{definition}
\label{def_fOmegaP}
For finite $P\subset\Omega^\circ$, set
\[
  f_{\Omega,P}=G_P^\Omega0_\Omega.
\]
Thus $f_{\Omega,P}$ is the least $\Omega$-tropical series whose corner
locus contains $P$.
\end{definition}

\begin{lemma}
\label{lemma_levelsets}
Let $f$ be a nonzero $\Omega$-tropical series and let
$0<\eta<\max_\Omega f$.  Then
\[
  \Omega_\eta=\{z\in\Omega:f(z)\geq\eta\}
\]
is a $\QQ$-polytope with nonempty interior, and $f|_{\Omega_\eta}$ is a
tropical polynomial.  If
$\eta<\min_{p\in P}f_{\Omega,P}(p)$, then
$P\subset\Omega_\eta^\circ$.
\end{lemma}

\begin{proof}
Continuity and the boundary condition imply
$\Omega_\eta\Subset\Omega^\circ$.  Lemma~\ref{lemma_estimate}, applied to
a compact neighborhood of $\Omega_\eta$, gives a finite set $\B$ containing
every canonical slope that can attain $f$ there.  Hence
$f=\min_{q\in\B}(q\cdot z+a_q(f))$ on that neighborhood.

Choose a rational-slope polytope $R$ with
$\Omega_\eta\subset R\Subset\Omega^\circ$ inside the same neighborhood.
Then
\[
 \Omega_\eta=
 R\cap\bigcap_{q\in\B}
 \{z:q\cdot z+a_q(f)\geq\eta\}.
\]
Indeed, a point of $R$ outside the superlevel set has $f<\eta$, so one of
the displayed monomials violates its inequality.  This is a finite
intersection of half-spaces with integral normals.  Finally,
$\{f>\eta\}$ is nonempty, which proves that the polytope has nonempty
interior.  The assertion about $P$ is immediate.
\end{proof}

\begin{proposition}[Continuity under exhaustion]
\label{prop_domain_exhaustion}
Let $Q_j\subset\Omega^\circ$ be $\QQ$-polytopes such that every compact
subset of $\Omega^\circ$ is contained in $Q_j$ for all sufficiently large
$j$, and assume that $P\subset Q_j^\circ$ for every $j$.  Put
\[
 F_j=G_P^{Q_j}0_{Q_j},\qquad F=G_P^\Omega0_\Omega.
\]
Then $F_j\to F$ locally uniformly on $\Omega^\circ$.
\end{proposition}

\begin{proof}
The canonical coefficients of the zero series on a domain $D$ are
\[
  a_q^D(0_D)=-\inf_{z\in D}q\cdot z.
\]
Since $Q_j\subset\Omega$, these coefficients are no larger for $Q_j$ than
for $\Omega$.  The value-vector map
\[
  (Tx)_q=\max\{x_q,\inf_{r\neq q}x_r\}
\]
used in the proof of Lemma~\ref{lem_distance} is coordinatewise monotone.
Apply the same fair one-point iteration on the two domains and pass to its
limit.  We obtain
\[
  a_q^{Q_j}(F_j)\leq a_q^\Omega(F)\qquad(q\in\ZZ^n),
  \qquad 0\leq F_j\leq F\quad\text{on }Q_j.
\]

On a fixed compact subset of $\Omega^\circ$, Lemma~\ref{lemma_estimate}
now restricts all locally active slopes of the functions $F_j$ to one finite
set.  Their coefficients are bounded, so every subsequence has a further
subsequence converging uniformly on compact sets to a tropical series
$H$ with $0\leq H\leq F$.  The corner condition at each point of $P$ is
closed under this convergence: on a fixed small ball, only finitely many
slopes are possible, so a pair of active slopes can be held fixed along a
subsequence.  Thus $P\subset C(H)$.  The bound $H\leq F$ lets $H$ extend
continuously by zero to $\partial\Omega$, and
Lemma~\ref{lemma_usualisomegatropical} makes it an $\Omega$-tropical
series.  Hence $H\in V(\Omega,P,0_\Omega)$ and minimality gives $F\leq H$.
Therefore $H=F$.  Every convergent subsequence has the same limit, proving
local uniform convergence of the full sequence.
\end{proof}

\begin{lemma}[Small mild initial data]
\label{lemma_mild_seed}
Let $Q_0$ be a $\QQ$-polytope, let $K\Subset Q_0^\circ$, and let
$\tau,\sigma>0$.  Given a finite set $E\subset\ZZ^n$, there exist a
$\QQ$-polytope $Q$, a $Q$-tropical polynomial $g$, and a finite set
$\B\subset\ZZ^n$ such that
\[
 K\subset Q^\circ\subset Q\subset Q_0,
 \qquad d_{\mathrm H}(Q,Q_0)<\sigma,
 \qquad 0\leq g\leq\tau,
\]
$E\subset\B=\operatorname{conv}(\B)\cap\ZZ^n$, and every
$q\in\B$ is uniquely minimal for $g$ on a nonempty open subset of
$Q^\circ$.  In particular, $g$ has only mild singularities.
\end{lemma}

\begin{proof}
Write
$Q_0=\{z:L_j(z)=q_j\cdot z+c_j\geq0,\ 1\leq j\leq s\}$ and set
\[
  u(z)=\min\{\tau,L_1(z),\ldots,L_s(z)\}.
\]
Then $Q_0=\{u\geq0\}$, the function $u$ vanishes on
$\partial Q_0$, and it is positive in $Q_0^\circ$.  Let
\[
  \B=\operatorname{conv}\bigl(E\cup\{0,q_1,\ldots,q_s\}\bigr)
      \cap\ZZ^n
\]
and assign to every $q\in\B$ its canonical coefficient for $u$ on $Q_0$.
Every lifted point $(q,a_q(u))$ lies on a lower face.

The coefficient function $q\mapsto a_q(u)$ is convex on $\B$.  Adding an
arbitrarily small positive multiple of $|q|^2$, followed by a generic still
smaller perturbation, gives a strictly convex lifting whose lower
subdivision is a triangulation using every point of $\B$.  Subtracting one
common constant does not change this subdivision, and the constant may be
chosen so that every total coefficient perturbation is strictly negative.
Let $g$ be the resulting tropical polynomial and put
\[
  Q=\{z:g(z)\geq0\}.
\]
Then $g\leq u$, so $Q\subset Q_0$.  If the perturbation is small enough,
compactness gives $K\subset Q^\circ$, $d_{\mathrm H}(Q,Q_0)<\sigma$, and
$g\leq\tau$ on $Q$.  The set $Q$ is a finite intersection of integral-normal
half-spaces and $g=0$ on its boundary.

The perturbations can also be chosen so that the full-dimensional dual cell
of every $q\in\B$ meets $Q^\circ$.  To see this, first note that the
unperturbed $q$-monomial touches $u$ at some point of $Q_0$.  A sufficiently
small strict downward perturbation creates its dual cell in an arbitrarily
small neighborhood of this contact set.  At an interior contact the claim is
immediate.  At a boundary contact, move a short distance into $Q_0$: the
unperturbed function becomes positive, while the perturbed monomial remains
the unique minimum before leaving that new dual cell.  Since the
configuration is finite, the perturbations may be chosen simultaneously.
Thus every $q\in\B$ is uniquely minimal somewhere in $Q^\circ$.  Finally,
$\B$ is saturated, so Remark~\ref{rem_mild} proves mildness.
\end{proof}

Compact convex domains admit inscribed $\QQ$-polytope exhaustions: take
convex hulls of increasingly fine finite sets of rational points in
$\Omega^\circ$.  Applying Lemma~\ref{lemma_mild_seed} with perturbations
tending to zero produces another such exhaustion by domains carrying
arbitrarily small mild initial data.

\section{The main theorem}
\label{sec_main}

Below we state and prove a theorem which initially was the main purpose of the present paper. On the one hand, this is the global essential ingredient for establishing the Abelian-to-Tropical scaling limit theorem in higher dimensions \cite{us}, with the other being \cite{nikita_hdsolitons} -- note that they fit each other perfectly as polyhedra with no lattice points except for their vertices, arising as local dual cells to the points on tropical hypersurfaces in the theorem below, are exactly the ones for which ``husking'' procedure stabilizes giving rise to higher dimensional analogues of solitons, triads and nodes of \cite{us_solitons}. On the other hand, it seems that it can be used in generalizing a recent development of Tropical-to-Affine scaling limit \cite{KLSS26} in two non-trivial ways: it provides a suitable polyhedral exhaustion of a domain, and gives a description of possible dual polyhedra to the vertices of the corresponding tropical hypersurfaces arising from an approximate relaxation, possibly of very high volumes. Thus, to derive the Monge-Amp\`ere equation from tropical relaxation, one would need to perform a subtler analysis of the distributions of vertices weighted by such volumes with respect to the perturbation points -- this new scaling limit would arise from the asymptotic equivalence of these two distributions.

\begin{theorem}[Approximation by a mild shrinking path]
\label{thm_main_mild}
Let $\Omega\subset\RR^n$ be a compact convex domain, let
$P\subset\Omega^\circ$ be finite, and put $F=G_P^\Omega0_\Omega$.
For every compact $K$ with $P\subset K\Subset\Omega^\circ$ and every
$\e>0$, there exist a $\QQ$-polytope
$K\subset Q^\circ\subset Q\subset\Omega^\circ$, a $Q$-tropical polynomial
$g$ with $0\leq g<\e$ and only mild singularities, and a finite sequence
\[
  h_0=g,h_1,\ldots,h_N
\]
with the following properties.
For each $i$, either $h_i=h_{i-1}$, or there are
$p_i\in P$, $q_i\in\ZZ^n$, $c_i>0$, and $0\leq\lambda_i<1$ such that
\[
 G_{p_i}^Qh_{i-1}=\Add_{q_i}^{c_i}h_{i-1},
 \qquad
 h_i=\Add_{q_i}^{\lambda_i c_i}h_{i-1}.
\]
Every tropical hypersurface along each interpolating segment
$\Add_{q_i}^{t c_i}h_{i-1}$, $0\leq t\leq\lambda_i$, has only mild
singularities on $K$, and
\[
  \sup_{z\in K}|h_N(z)-F(z)|<\e.
\]
\end{theorem}

\begin{proof}
Put $m=|P|$ and $D=\operatorname{diam}(\Omega)$.  Choose an inscribed
$\QQ$-polytope exhaustion $Q_j^0$ of $\Omega$ and a compact exhaustion
$C_j$ such that
\[
  K\subset C_j\Subset(Q_j^0)^\circ.
\]
After passing to a subsequence, we may suppose that the $C_j$ exhaust
$\Omega^\circ$.  Set
\[
 r_j=\frac12\operatorname{dist}
      (C_j,\RR^n\setminus(Q_j^0)^\circ),\qquad
 L_j=\{z:\operatorname{dist}(z,C_j)\leq r_j\},
\]
and decrease $r_j$ if necessary so that $L_j\Subset(Q_j^0)^\circ$.
Let
\[
 A=1+mD,\qquad
 E_j=\{q\in\ZZ^n:|q|\leq A/r_j\}.
\]
Apply Lemma~\ref{lemma_mild_seed} to $Q_j^0,L_j$, the set $E_j$, and
numbers $\tau_j,\sigma_j\downarrow0$ with $\tau_j<1$.  Denote the resulting
objects by $Q_j,g_j,\B_j$.  Since $L_j\subset Q_j^\circ$, the domains
$Q_j$ still exhaust $\Omega^\circ$, and every point of $C_j$ has an
$r_j$-ball contained in $Q_j^\circ$.

Proposition~\ref{prop_domain_exhaustion} now allows us to choose $j$ so
large that, on writing $Q=Q_j$, $g=g_j$, $\B=\B_j$, and $r=r_j$, we have
$\tau_j<\e/4$ and
\[
  \|G_P^Q0_Q-F\|_K<\e/4.
\]
In particular,
$\|g\|_{\infty,Q}\leq\tau_j<\e/4$.  Lemma~\ref{lem_finite_coefficient_lipschitz}
gives
\[
  \rho_Q(g,0_Q)\leq\|g\|_{\infty,Q}<\e/4.
\]
Thus Corollary~\ref{cor_gp_nonexpansive} yields
\[
  \|G_P^Qg-G_P^Q0_Q\|_{\infty,Q}<\e/4.
\]

By Theorem~\ref{thm_finite_stabilization}, a finite sequence of exact
one-point updates carries $g$ to $G_P^Qg$.  Write it as
\[
 f_0=g,\qquad f_i=G_{p_i}^Qf_{i-1},\qquad
 f_N=G_P^Qg,
\]
omitting repetitions after stabilization.  We construct a nearby mild path
inductively.  Set $h_0=g$ and choose positive numbers $\kappa_i$ with
$\sum_i\kappa_i<\e/4$.  Given $h_{i-1}$, put
$\widehat h_i=G_{p_i}^Qh_{i-1}$.  Nonexpansiveness gives
\[
 \rho_Q(\widehat h_i,f_i)
 \leq\rho_Q(h_{i-1},f_{i-1}).
\]
If the update is trivial, take $h_i=h_{i-1}$.  Otherwise write
$\widehat h_i=\Add_{q_i}^{c_i}h_{i-1}$ and choose
$\lambda_i<1$ sufficiently close to $1$ that
\[
 \rho_Q(\Add_{q_i}^{\lambda_i c_i}h_{i-1},\widehat h_i)<\kappa_i.
\]
It follows inductively that
$\rho_Q(h_N,G_P^Qg)<\sum_i\kappa_i<\e/4$.

It remains to verify the local mildness assertion.  Every function on every
partial segment is bounded above by $G_P^Qg$.  Indeed, this is true initially,
and monotonicity gives
\[
 G_{p_i}^Qh_{i-1}\leq G_{p_i}^QG_P^Qg=G_P^Qg;
\]
a partial update lies between its initial and terminal functions.  By
Proposition~\ref{prop_upperbound}, for $z\in Q$,
\[
 G_P^Qg(z)\leq g(z)+m l_Q(z)\leq1+mD=A,
\]
where $e_1$ is the first standard basis vector, because
$$l_Q\leq l_Q^{e_1}\text{
\ and\ }
 l_Q^{e_1}(z)\leq\operatorname{diam}(Q)\leq D.$$

Now let a canonical slope $q$ be active at some $z\in K$ anywhere along
one of the full segments, including its endpoint.  Its affine monomial is
nonnegative on $Q$ and has value at most $A$ at $z$.  Since the $r$-ball
about $z$ is contained in $Q^\circ$, the estimate in the proof of
Lemma~\ref{lemma_estimate} gives $|q|\leq A/r$.  Hence
$q\in E_j\subset\B$.  Lemma~\ref{lemma_mild_seed} says that every slope in
$\B$ is initially uniquely minimal somewhere in $Q^\circ$.  Applying
Lemma~\ref{lemma_mild} inductively therefore shows that this remains true
and that every interpolating hypersurface has only mild singularities on
$K$.

Finally, the three estimates above give
\[
 \|h_N-F\|_K
 \leq \|h_N-G_P^Qg\|_K
      +\|G_P^Qg-G_P^Q0_Q\|_K
      +\|G_P^Q0_Q-F\|_K
 <\e.
\]
\end{proof}

\section{Declarations}

\subsection{Ethical Approval} Not applicable.

\subsection{Funding} M.S. is supported by the Simons Foundation, grant [SFI-MPS-T-Institutes-00007697], the Ministry of Education and Science of the Republic of Bulgaria, grant [DO1-239/10.12.2024], as well as by the National Science Fund,
the Ministry of Education and Science of the Republic of Bulgaria, under contract [KP-06-N92/2].


\begin{thebibliography}{10}

\bibitem{Abakumov2017}
\newblock E.~Abakumov and E.~Doubtsov,
\newblock Approximation by proper holomorphic maps and tropical power series,
\newblock \emph{Constr. Approx.}, \textbf{47} (2018), no.~2, 321--338,
\newblock \urlprefix\url{https://doi.org/10.1007/s00365-017-9375-5}.

\bibitem{Brug}
\newblock E.~Brugall{{\'e}},
\newblock Some aspects of tropical geometry,
\newblock \emph{Eur. Math. Soc. Newsl.}, no.~83 (2012), 23--28.

\bibitem{BIMS}
\newblock E.~Brugall{\'e}, I.~Itenberg, G.~Mikhalkin and K.~Shaw,
\newblock Brief introduction to tropical geometry,
\newblock in \emph{Proceedings of the {G}{\"o}kova {G}eometry-{T}opology
  {C}onference 2014},
\newblock G{\"o}kova Geometry/Topology Conference (GGT), G{\"o}kova, 2015,
\newblock 1--75.

\bibitem{firstsand}
\newblock S.~Caracciolo, G.~Paoletti and A.~Sportiello,
\newblock Conservation laws for strings in the abelian sandpile model,
\newblock \emph{EPL (Europhysics Letters)}, \textbf{90} (2010), 60003,
\newblock \urlprefix\url{https://doi.org/10.1209/0295-5075/90/60003}.

\bibitem{MR2482129}
\newblock R.~G. Halburd and N.~J. Southall,
\newblock Tropical {N}evanlinna theory and ultradiscrete equations,
\newblock \emph{Int. Math. Res. Not. IMRN}, 2009, no.~5, 887--911,
\newblock \urlprefix\url{https://doi.org/10.1093/imrn/rnn150}.

\bibitem{nikita_hdsolitons}
N.~Kalinin.
\newblock Sandpile Solitons in Higher Dimensions,
\newblock \emph{Arnold Mathematical Journal}
\newblock 9, no. 3 (2023), 435-454
\newblock \urlprefix\url{https://doi.org/10.1007/s40598-023-00224-7}

\bibitem{announce}
\newblock N.~Kalinin and M.~Shkolnikov,
\newblock Tropical curves in sandpiles,
\newblock \emph{C. R. Math. Acad. Sci. Paris}, \textbf{354} (2016),
  125--130,
\newblock \urlprefix\url{https://doi.org/10.1016/j.crma.2015.11.003}.

\bibitem{us_solitons} N.~Kalinin, and M.~Shkolnikov,
\newblock Sandpile solitons via smoothing of superharmonic functions.
\newblock \emph{Communications in Mathematical Physics}
\newblock 378, no. 3 (2020),
\newblock 1649--1675,
\newblock \urlprefix\url{https://doi.org/10.1007/s00220-020-03828-8}.

\bibitem{sandcomputation}
\newblock N.~Kalinin, A.~Guzm{\'a}n-S{\'a}enz, Y.~Prieto, M.~Shkolnikov,
  V.~Kalinina and E.~Lupercio,
\newblock Self-organized criticality and pattern emergence through the lens of
  tropical geometry,
\newblock \emph{Proceedings of the National Academy of Sciences}, \textbf{115}
  (2018), no.~35, E8135--E8142,
\newblock \urlprefix\url{https://doi.org/10.1073/pnas.1805847115}.

\bibitem{us}
\newblock N.~Kalinin and M.~Shkolnikov,
\newblock Tropical curves in sandpile models,
\newblock  \emph{Advances in Mathematics}, 502 (2026): 111127,
\newblock \urlprefix\url{https://doi.org/10.1016/j.aim.2026.111127}.

\bibitem{us_series}
\newblock N.~Kalinin and M.~Shkolnikov,
\newblock Introduction to tropical series and wave dynamic on them,
\newblock \emph{Discrete \& Continuous Dynamical Systems-A}, \textbf{38}
  (2018), no.~6, 2827--2849,
\newblock \urlprefix\url{https://doi.org/10.3934/dcds.2018120}.

\bibitem{KLSS26}
\newblock N.~Kalinin, E.~Lupercio, H.~Serrano, and M.~Shkolnikov,
\newblock Many-point tropical relaxation and the Monge-Amp\`ere equation,
\newblock preprint (2026), \arxiv{2607.25878}.

\bibitem{LS26}
\newblock E.~Lupercio, and M.~Shkolnikov,
\newblock Symmetry Emergence in Self-Organized Criticality,
\newblock preprint (2026), \arxiv{2608.13500}.


\bibitem{MR743626}
\newblock C.~O. Kiselman,
\newblock Croissance des fonctions plurisousharmoniques en dimension infinie,
\newblock \emph{Ann. Inst. Fourier (Grenoble)}, \textbf{34} (1984), 155--183,
\newblock \urlprefix\url{https://www.numdam.org/item?id=AIF_1984__34_1_155_0}.

\bibitem{kiselman2014questions}
\newblock C.~O. Kiselman,
\newblock Questions inspired by {M}ikael {P}assare's mathematics,
\newblock \emph{Afrika Matematika}, \textbf{25} (2014), no.~2, 271--288,
\newblock \urlprefix\url{https://doi.org/10.1007/s13370-012-0107-5}.

\bibitem{korhonen2015tropical}
\newblock R.~Korhonen, I.~Laine and K.~Tohge,
\newblock \emph{Tropical value distribution theory and ultra-discrete
  equations},
\newblock World Scientific, 2015.

\bibitem{lahaye2015compositions}
\newblock S.~Lahaye, J.~Komenda and J.-L. Boimond,
\newblock Compositions of (max,+) automata,
\newblock \emph{Discrete Event Dynamic Systems}, \textbf{25} (2015),
  no.~1--2, 323--344,
\newblock \urlprefix\url{https://doi.org/10.1007/s10626-014-0186-6}.

\bibitem{MR2795727}
\newblock I.~Laine and K.~Tohge,
\newblock Tropical {N}evanlinna theory and second main theorem,
\newblock \emph{Proc. Lond. Math. Soc. (3)}, \textbf{102} (2011), 883--922,
\newblock \urlprefix\url{https://doi.org/10.1112/plms/pdq049}.

\bibitem{lombardy2006sequential}
\newblock S.~Lombardy and J.~Sakarovitch,
\newblock Sequential?,
\newblock \emph{Theoret. Comput. Sci.}, \textbf{356} (2006), no.~1--2,
  224--244,
\newblock \urlprefix\url{https://doi.org/10.1016/j.tcs.2006.01.028}.

\bibitem{mikh2}
\newblock G.~Mikhalkin,
\newblock Tropical geometry and its applications,
\newblock in \emph{International {C}ongress of {M}athematicians. {V}ol. {II}},
\newblock Eur. Math. Soc., Z{\"u}rich, 2006,
\newblock 827--852.

\bibitem{mikhalkin2023wave}
\newblock G.~Mikhalkin and M.~Shkolnikov,
\newblock Wave fronts and caustics in the tropical plane,
\newblock in \emph{Proceedings of the {G}\"okova {G}eometry-{T}opology
  {C}onference 2023}, International Press, 2024, 11--48,
\newblock \arxiv{2310.17269}.

\bibitem{mikhalkin2006tropical}
\newblock G.~Mikhalkin and I.~Zharkov,
\newblock Tropical curves, their {J}acobians and theta functions,
\newblock in \emph{Curves and abelian varieties}, vol. 465 of Contemp. Math.,
\newblock Amer. Math. Soc., Providence, RI, 2008,
\newblock 203--230,
\newblock \urlprefix\url{https://doi.org/10.1090/conm/465/09104}.

\bibitem{shkolnikov2023relaxation}
\newblock M. Shkolnikov, 
\newblock Relaxation in one-dimensional tropical sandpile,
\newblock in \emph{Communications in Mathematics},
\newblock 31, 2023, no. 3, 
\newblock 21--31,
\newblock \urlprefix\url{https://doi.org/10.46298/cm.10483}.

\bibitem{tohge2014order}
\newblock K.~Tohge,
\newblock The order and type formulas for tropical entire functions---another
  flexibility of complex analysis,
\newblock in \emph{Complex Analysis and its Applications to Differential and
  Functional Equations}, Reports and Studies in Forestry and Natural
  Sciences, vol.~14, University of Eastern Finland, 2014, 113--164.

\bibitem{vafa2012supersymmetric}
\newblock C.~Vafa,
\newblock Supersymmetric partition functions and a string theory in 4
  dimensions,
\newblock preprint (2012), \arxiv{1209.2425}.

\bibitem{Varypaev25}
\newblock A.~ Varypaev, 
\newblock On the probability distribution of relaxation length in a parametric one-dimensional tropical sandpile model
\newblock In \emph{Proceedings of the Bulgarian Academy of Sciences}, 2025
\newblock Vol. 78, No. 8, pp. 1121-1131,
\newblock \urlprefix\url{https://doi.org/10.7546/CRABS.2025.08.01}.

\end{thebibliography}

\providecommand{\href}[2]{#2}
\providecommand{\arxiv}[1]{\href{https://arxiv.org/abs/#1}{arXiv:#1}}
\providecommand{\url}[1]{\texttt{#1}}
\providecommand{\urlprefix}{URL }

\vspace{27pt}
{\centering
 Guangdong Technion Israel Institute of Technology (GTIIT),\\
241 Daxue Road, Shantou, Guangdong Province 515603, P.R. China,\\  Technion-Israel Institute of Technology, Haifa, 32000, Haifa District, Israel,\\ nikaanspb@gmail.com\\
\vspace{12pt}
Institute of Mathematics and Informatics\\ at the Bulgarian Academy of Sciences,\\
akad. Georgi Bonchev 8, 1113, Sofia, Bulgaria,\\
m.shkolnikov@math.bas.bg

}
\end{document}